\documentclass[
 aip,
 amsmath,amssymb,
reprint, onecolumn 
]{revtex4-1}

\usepackage{graphicx}
\usepackage{dcolumn}
\usepackage{bm}

\usepackage[utf8]{inputenc}
\usepackage[T1]{fontenc}
\usepackage{mathptmx}
\usepackage{etoolbox}

\usepackage{amsmath}
\usepackage{amsfonts}
\usepackage{graphicx}
\graphicspath{{/texfiles/}}

\usepackage{amsmath,amsthm}
\usepackage{lineno,hyperref}
\modulolinenumbers[5]
\usepackage{geometry,amsmath,amsfonts,amssymb,mathrsfs,epsfig,bm} 
\usepackage{amstext}
\numberwithin{equation}{section}

\newtheorem{thm}{Theorem}[section]
\newtheorem{cor}[thm]{Corollary}
\newtheorem{lem}[thm]{Lemma}
\newtheorem{prop}[thm]{Proposition}

\def\Z{\mathbb{Z}}
\def\Q{\mathbb{Q}}
\def\R{\mathbb{R}}

\def\P{\mathbb{P}}
\def\E{\mathbb{E}}

\def\LL{\mathcal{L}}

\def\Spin{\mathsf S}
\renewcommand{\phi}{\varphi}
\renewcommand{\epsilon}{\varepsilon}
\newcommand{\1}{{\text{\Large $\mathfrak 1$}}}

\newcommand{\cov}{\operatorname{cov}}

\renewcommand{\liminf}{\varliminf}

\def\I{\textbf{I}}
\def\II{\textbf{I\!I}}

\def\dd{\bm d}
\def\ent{\operatorname{\mathsf{Ent}}}

\newcommand{\III}[1]{{\left\vert\kern-0.5ex\left\vert\kern-0.5ex\left\vert #1 
    \right\vert\kern-0.5ex\right\vert\kern-0.5ex\right\vert}}

\def\Z{\mathbb{Z}}
\def\Q{\mathbb{Q}}
\def\R{\mathbb{R}}

\def\P{\mathbb{P}}
\def\E{\mathbb{E}}

\def\LL{\mathcal{L}}

\newcommand{\He}{\mathbb{H}}

\usepackage{color}

\makeatletter
\def\@email#1#2{%
 \endgroup
 \patchcmd{\titleblock@produce}
  {\frontmatter@RRAPformat}
  {\frontmatter@RRAPformat{\produce@RRAP{*#1\href{mailto:#2}{#2}}}\frontmatter@RRAPformat}
  {}{}
}%
\makeatother
\begin{document}

\preprint{AIP/123-QED}

\title[The log-Sobolev  inequality for spin systems of higher order interactions.]{The log-Sobolev  inequality for spin systems of higher order interactions.}
\author{Takis Konstantopoulos}
\email{T.Konstantopoulos@liverpool.ac.uk}
\affiliation{ 
  Department of Mathematical Sciences,   University of Liverpool, Mathematical Sciences Building, Liverpool, L69 7ZL, United Kingdom..
}
\author{Ioannis Papageorgiou}
 \email{ papyannis@yahoo.com,  i.papageorgiou@ufabc.edu.br}
 \affiliation{     Centro de Matemática, Computação e Cognição (CMCC), Universidade Federal do ABC (UFABC), Avenida dos Estados, 5001 - Santo Andre - Sao Paulo, Brasil.
}%
\date{\today}

\begin{abstract}
We study the infinite-dimensional log-Sobolev inequality for spin systems
on  $\mathbb{Z}^d$ with interactions of power higher than quadratic.
We assume that the one site measure without a boundary $e^{-\phi(x)}dx/Z$    satisfies a log-Sobolev inequality and  we determine conditions  so that  the infinite-dimensional Gibbs measure also satisfies  the inequality.
As a concrete application, we prove that a
certain class of nontrivial Gibbs measures with non-quadratic
interaction potentials on an infinite product of Heisenberg groups satisfy the
log-Sobolev inequality.
\end{abstract}

\maketitle

\section{Introduction}

Coercive inequalities, like the logarithmic Sobolev, play an important role in the study of ergodic properties of stochastic systems.  The inequality is  associated with strong properties about the type and speed of convergence of Markov semigroups to invariant measures. In particular, in the field  of infinite dimensional  interacting spin systems, they provide a powerful tool in the examination of  the infinite volume Gibbs measures. In the current paper we give a first explicit description of spin systems with interactions that are higher than quadratic that satisfy the log-sobolev inequality,  and thus provide a first example in the bibliography of spin systems  with high order interactions that converge exponentially fast to equilibrium.

Our focus is on the  typical logarithmic Sobolev (abbreviated as log-Sobolev or LS) 
inequality  for probability measures governing systems of unbounded spins on  
the $d$-dimensional lattice $\mathbb{Z}^d$
with nearest neighbour interactions of order higher than 2.  
The aim of this paper is to investigate conditions on the local specification function
so that the inequality can be extended from the single-site interaction free  measure
to the infinite-dimensional Gibbs measure, assuming that the latter exists.
One crucial assumption  is that the single-site without interactions (consisting only of the phase) measure satisfies a  log-Sobolev inequality.  In addition, we assume that the power of the interaction is 
dominated by that of the phase. 
As an application, we show that the log-Sobolev inequality holds for the infinite Gibbs measure on spin systems 
with values in  the  Heisenberg group $\He_1$.

The single-site space will be denoted by $\Spin$ 
(colloquially, ``spins take values in $\Spin$'')
and $\Omega:={\Spin}^{\Z^d}$.
For $\Lambda$ a finite subset of $\Z^d$,
denote by $\P^{\Lambda,\omega}$ a probability measure on $\Spin^\Lambda$
that depends on the boundary conditions $\omega \in \Spin^{\partial \Lambda}$.
These probability measures (known as local specifications)
satisfy the usual spatial Markov property which imposes
sever restrictions on them, namely, they must, under natural
assumptions, be of  Gibbs type with a Hamiltonian that
can be split into two parts: the phases (depending on single
sites) and the interaction (depending on neighboring sites).
Denote by $\E^{\Lambda,\omega}$ integration with respect 
to $\P^{\Lambda,\omega}$;
and use the  convention that the former symbol be used in place of the
latter; see, e.g., Guionnet and Zegarlinski \cite{G-Z}.
Critiria for a measure (also a   local specification with quadratic interactions) to
satisfy a log-Sobolev inequality uniformly has been investigated by
Zegarlinski \cite{Z2},
Bakry and Emery \cite{B-E}, 
Yoshida \cite{Y}, An\'e {\em et al.} \cite{A-B-C},  
Bodineau and Helfer \cite{B-H}, Ledoux \cite{Led}, Helfer \cite{H} and  Baudoin and Bonnefont \cite{BAUDOIN20122646}.
Furthermore, in Gentil and Roberto \cite{G-R} the spectral gap inequality 
is proved, while     Gentil,  Guillin and  Miclo,  in 
 \cite{Ge-Gu-M-05} and   \cite{Ge-Gu-M-07},    Gozlan, Roberto and  Samson in  \cite{Go-Ro-Sa13}, Barki, Bobkov, Dagher and Roberto in \cite{B-B-D-R24} and Papageorgiou in \cite{Pa5}  studied  the modified log-sobolev inequality.  p-logarithmic Sololev inequalities have been studied by
	  Balogh,    Engulatov,  Hunziker and   Maasalo  in \cite{Ba-En-Hu-Ma12}, Balogh,   Don and  Kristaly in \cite{Ba-Do-Kr24} and Balogh,    Kristaly and   Tripaldi in \cite{BALOGH2024110217}. Lyapunov methods  have been used to prove log-Sobolev inequalities by Feng and Li \cite{Feng-Li23a}, \cite{Feng-Li23b}, while the relation of the inequality with the  logarithmic Schrödinger equation was studied by Read, Zegarlinski and Zhang \cite{Re-Ze-Zh22}. 
 
For the single-site measure  on the real line
with or without  boundary conditions 
necessary and sufficient  conditions in order that
the log-Sobolev inequality be satisfied uniformly over the boundary
conditions $\omega$ are presented in  Bobkov and G\"otze \cite{B-G}, 
Bobkov and Zegarlinski \cite{B-Z} and Roberto and Zegarlinski \cite{R-Z}.

The log-Sobolev inequality for the infinite-dimensional Gibbs measure on the
lattice is examined in Guionnet and Zegarlinski \cite{G-Z}, 
and Zegarlinski \cite{Z1}, \cite{Z2}. 
The problem of passing from single-site to infinite-dimensional measure,
in presence of quadratic interactions,
is addressed by Marton \cite{M1}, Inglis and Papageorgiou \cite{I-P}, 
Otto and Reznikoff \cite{O-R} and Papageorgiou \cite{Pa3}.

Working beyond the case of quadratic interactions is the scope of
this paper.
Non-quadratic interactions have been considered in \cite{Pa2}, but for the
case of the one-dimensional lattice and the stronger log-Sobolev $q$-inequality.
In that paper, the inequality for the infinite-dimensional Gibbs measure 
was related to the inequality for the
finite projection of the Gibbs measure. In \cite{I-P1} conditions have been investigated so that the infinite dimensional Gibbs measure satisfies the inequality under the main assumption that the single-site measure satisfies a log-Sobolev inequality uniformly on the boundary conditions. Under the same framework, concentration properties have been studied in \cite{Pa4}. Recentely, Limmer, Kratsios, Yang, Saqur and Horvath in \cite{LKYSH24} proved convergence properties under the assumption of log-sobolev inequalities.

The scope of the current paper is to  prove the log-Sobolev inequality for the Gibbs measure without setting conditions  neither on the local specification $\{\E^{\Lambda,\omega}\}$ nor on the one site measure $\mathbb{E}^{\{i\},\omega}$. What we actually show is that under appropriate conditions on the interactions,  the Gibbs measure  satisfies a log-Sobolev inequality whenever the boundary free one site measure $\mu(dx)=e^{-\phi(x)}dx /(\int e^{-\phi(x)} dx)$  satisfies a log-Sobolev inequality.
 In that way we improve the previous results since the log-Sobolev inequality is determined alone by the phase  $\phi$ of the simple without interactions measure $\mu$ on $M$, for which a plethora of criteria and examples of good measure that satisfy the inequality exist.

To explain the applicability of our general infinite-dimensional framework the
specific case of the Heisenberg group is presented. This will serve as a
specific example (see Theorem \ref{thmExample}) derived from the more general
result of Theorem \ref{thmGENERAL}.

\subsection{General framework}
Consider the $d$-dimensional integer lattice $\Z^d$ equipped with 
the standard neighborhood
structure: two lattice points (sites) $i, j \in \Z^d$ are neighbors (write $i \sim j$)
if $\sum_{1\le k \le d} |i_k-j_k|=1$. We shall be working with the configuration
space $\Omega = \Spin^{\Z^d}$ where $\Spin$ is an appropriate ``spin space''. We consider the spin space $\Spin$ to be a group, and we denote $\cdot$ the group operation and  $x^{-1}$ the inverse of $x\in \Spin$ in respect to the group operation.
The coordinate $\omega_i$ of a configuration $\omega \in \Omega$ is referred to
as the spin at site $i$; $\omega_i$ takes values in $\Spin^i \equiv \Spin$.
When $\Lambda \subset \Z^d$ we identify $\Spin^\Lambda$
with the Cartesian product of the $\Spin^i$ when $i$ ranges over $\Lambda$.
We assume that $\Spin$ comes with a natural measure; for example,
when $\Spin$ is a group then the measure is one which is invariant under
the group operation; we write $dx_i$ for this measure on the copy
$\Spin^i$ of $\Spin$ corresponding to site $i \in \Z^d$; 
and we use the symbol $dx_\Lambda$ for a product measure, that is,
the product of the $dx_i$, $i \in \Lambda$.    
It is assumed that $\E^{\{i\},\omega}$ is absolutely continuous
with respect to $dx_i$. The Markov property implies then that,
for finite subsets $\Lambda$ of $\Z^d$, the probability measures
$\E^{\Lambda, \omega}$ should be of a very special form (see \cite{Pr}):
\[
\E^{\Lambda, \omega} (dx_\Lambda) = \frac{1}{Z^{\Lambda,\omega}}\,
e^{-H^{\Lambda,\omega}(x_\Lambda)}\, dx_\Lambda,
\]
where $Z^{\Lambda,\omega}$ is a normalization constant
and where the function $H^{\Lambda, \omega}$ (the Hamiltonian) is of 
the form
\[
H^{\Lambda,\omega}(x_\Lambda) 
:= \sum_{i \in \Lambda} \phi(x_i) 
+ \sum_{i,j \in \Lambda,\, j \sim i} J_{ij} V(x_i, x_j)
+ \sum_{i\in \Lambda,j \in \partial \Lambda,\, j \sim i} J_{ij} V(x_i, \omega_j),
\]
the sum of the phase and the interactions.

It is implicitly assumed that the normalization constants are finite.
Several conventions are tacitly used in this business. 
When $f$ is a function from $\Spin^{\Z^d}$ into $\R$,
we let $\E^{\Lambda, \omega} f$ for the function
on $\Spin^{\Z^d}$ obtained by integrating $f$ with respect
to $dx_\Lambda$ and by substituting $x_{\partial \Lambda}$ by $\omega$,
while leaving all other coordinates the same.
When we simply write $\E^\Lambda f$ we shall understand this as 
above with $\omega = x_{\partial \Lambda}$. 
Thus, $\E^\Lambda$ can be thought of as a linear operator  that
takes functions on the whole of $\Spin^{\Z^d}$ to functions
that do not depend on the variables $x_i, i \in \Lambda$.
Similarly, we will write $H^\Lambda$ for the Hamiltonian  $H^{\Lambda,\omega}$.  If $\Lambda$ is an  infinite subset of $\Z^d$ with the property
that any two points in $\Lambda$ are at lattice distance strictly
greater than $1$ from one another then $\E^{\Lambda, \omega}$
is the product of $\E^{\{i\}, \omega_{\partial\{i\}}}$.
Using these conventions, the spatial Markov property can then 
be expressed as
\[
\E^\Lambda \E^K = \E^\Lambda, \quad K \subset \Lambda.
\]
The Markov property written in this way, following the conventions above,
carries a lot of weight: in particular, it entails  that the law
of $x_K$ given $x_{\partial \Lambda}$ is the
law of $x_K$ given $x_{\partial K}$ integrated over
$x_{\partial K}$ when the later has the law obtained from $\P^{\Lambda}$.
This Markov property can, naturally,   be seen to be 
equivalent to the usual Markov property for Markov
processes indexed by the one-dimensional lattice $\Z$ (which
is often interpreted as ``time'' in view of the natural
total order of $\Z$.)

We say that the probability measure $\nu$ on $\Omega = \Spin^{\Z^d}$
is an infinite volume Gibbs measure for  the local
specifications $\{\E^{\Lambda,\omega}\}$ if
the Dobrushin-Lanford-Ruelle equations are satisfied:
\[
\nu \mathbb{E}^{\Lambda,\bullet}=\nu, \quad \Lambda \Subset \Z^d,
\]
that is, if $\nu$ is an invariant measure for the Markov
random field.
We  refer to Preston \cite{Pr}, Dobrushin \cite{D} and  Bellisard and   Hoegn-Krohn \cite{B-HK} for details.
Throughout the paper we shall assume that we are in the case
where $\nu$ exists and is unique (although uniqueness can be
deduced from our main results).

We next make some assumptions about the nature of 
the spin space $\Spin$. 

We shall assume that $\Spin$ is a nilpotent  Lie group on
$\R^d$ with H\"ormander system $X^1, \ldots, X^n$, $n \le d$,
satisfying the following relation:
if $X^k = \sum_{j=1}^d a_{kj} \frac{\partial}{\partial x_j}$,
$k=1,\ldots, n$, then $a_{kj}$ is a function of $x \in \R^d$
not depending on the $j$-th coordinate $x_j$; that is,
if $x, y \in \R^d$ have $x_j=y_j$ then $a_{kj}(x)=a_{kj}(y)$.
The gradient $\nabla$ with respect to 
this system is the vector operator $\nabla f = (X^1 f, \ldots, X^n f)$,
whereas $\Delta = (X^1)^2 + \cdots + (X^n)^2$ is the
sublaplacian, where $(X^k)^2 f = X^k(X^k f)$. 
We let $\|\nabla f\|^2 := (X^1 f)^2 + \cdots + (X^n f)^2$ (for general hypocoercive-typeoperators type generators see Kontis, Ottobre   and  Zegarlinski  \cite{KONTIS1750015} and  \cite{KONTIS20163173}, as well as Lugiewicz and Zegarlinski \cite{Lugiewicz2007CoerciveIF}).
When these operators act
on functions on the spin space $\Spin^i$ at site $i \in \Z^d$ 
they will be denoted by $\nabla_i$ and $\Delta_i$, respectively. 
If $\Lambda$ is a finite subset of  $\Z^d$ we shall let 
$\nabla_\Lambda := (\nabla_i, i \in \Lambda)$ and 
$\|\nabla_\Lambda f\|^2 := \sum_{i \in \Lambda} \|\nabla_i f\|^2$.
We shall assume that $\Spin$ comes equipped with a metric-like
function $\dd(x,y)$, $x,y \in \Spin$.
For example, if $\Spin$ is a Euclidean space then $\dd$ is the Euclidean
metric. If $\Spin$ is the Heisenberg group, 
then $\dd$ is the Carnot-Carath\'eodory metric. 
More generally, the role of $\dd$ only appears through
the assumptions we make.

In each and every case, the notation $\dd(x)$, for $x \in \Spin$, stands
for $\dd(x, 0)$, where $0$ is a special point of $\Spin$, for
example the origin if $\Spin$ is $\R^m$ or the identity element
if $\Spin$ is a Lie group.

The main assumption of the paper is that  the single site measure without interactions  (consisting only of the phase) 
$$\mu(dx)=\frac{e^{-\phi(x)}dx}{\int e^{-\phi(x)} dx}$$  
satisfies the log-Sobolev inequality, that is,
that there exists $c>0$ such that
\[
\mu\bigg( f^2 \log \frac{f^2}{\mu f^2}
\bigg) \le c\, \mu \|\nabla f\|^2
\]
for any smooth function $f: \Spin \mathsf \to \R$ such  that
 both sides make sense.

When the last inequality holds for $\E^{\Lambda,\omega}$ in the place of $\mu$ for the constant $c$ uniformly on the boundary conditions $\omega$,  we say  that the log-Sobolev inequality holds for
$\E^{\Lambda,\omega}$ {\em uniformly} (in $\omega$.)

We point out that when two measures satisfy the log-Sobolev inequality then their product also satisfies the inequality.   Similar thing is also true for 
spectral gap inequalities (a measure $\mu$ satisfies spectral gap inequality
with constant $C$ if
$\mu |f-\mu f|^2 \le C\, \mu |\nabla f|^2$).

Proofs of these assertions can be found in  
Gross \cite{G}, Guionnet and Zegarlinski \cite{G-Z} and 
Bobkov and Zegarlinski \cite{B-Z}. 
In that way, if for every $i\in \Lambda$, $\E^{\{ i \},\omega}$ satisfies the log-Sobolev (similarly the Spectral gap) inequality uniformly and  $\Lambda$ is a subset (finite or infinite) of $\Z^d$ such that
any two points of $\Lambda$ are at lattice distance strictly greater than one from
one another, then the log-Sobolev (similarly spectral gap) inequality holds for
$\E^{\Lambda,\omega}$, with the same constant $c$, uniformly
in $\omega \in \partial \Lambda$.

\subsection{The Heisenberg group}\label{introHeis}
The Heisenberg group $\He_1$ can be identified with 
$\mathbb{R}^3$ equipped with the group operation
$$
x\cdot\tilde{x} = 
\big(x_1 + \tilde{x}_1, x_2 + \tilde{x}_2, x_3 + \tilde{x}_3 
+ \frac{1}{2}(x_1\tilde{x}_2 - x_2\tilde{x}_1)\big).
$$
It is a Lie group with Lie algebra 
which can be identified with the space of left-invariant vector 
fields on $\He_1$ in the standard way.  
See, e.g., \cite{B-L-U}.
By direct computation, the
vector fields
\begin{eqnarray}
X_1 &=& \partial_{x_1} - \frac{1}{2}x_2\partial_{x_3} \nonumber \\
X_2 &=& \partial_{x_2} + \frac{1}{2}x_1\partial_{x_3} \nonumber \\
X_3 &=& \partial_{x_3} = [X_1, X_2] \nonumber,
\end{eqnarray}
where $\partial_{x_i}$ denoted derivation with   respect to $x_i$, form a Jacobian basis.
From this it is clear that $X_1, X_2$ satisfy the H\"ormander  condition 
(i.e., $X_1, X_2$ and their commutator $[X_1, X_2]$ span the tangent space 
 at  every point of $\He_1$).   It is also 
easy to check that the left-invariant Haar measure 
(being also right-invariant measure owing to the fact that 
the group is nilpotent) 
is the Lebesgue measure on $\R^3$.

The gradient  is given by
$
\nabla := (X_1,X_2),
$
and  the \textit{sub-Laplacian}  by 
$
\Delta := X_1^2 + X_2^2.
$ 
A probability measure $\mu$ on $\He_1$
satisfies a log-Sobolev inequality if there exists
a positive constant $c$ such that

\begin{equation*}
\mu\left(
f^{2}\log\frac{f^{2}}{\mu f^{2}}\right) \leq c\, \mu \|\nabla f\|^2= c \mu \left((X_1f)^2+(X_2 f)^2\right),
\end{equation*}
for all smooth  functions $f: \He_1 \to \R$. Here, $\mu(g)$, or,
simply, $\mu g$
stands for $\int_{\He_1} g \, d\mu$.
The quantity on the left-hand  side is the $\mu$-entropy
of the function $f^2$ or, equivalently, the Kullback-Leibler
divergence between the measure $f^2 d\mu$ and $\mu$.
For example, the family of measures
\begin{align}
\label{introUb2}
\mu_p (dx) 
:= \frac{e^{-\beta \dd(x,e)^p}}{\int_{\He_1} e^{-\beta \dd(x,e)^p}dx}\,dx,
\end{align}
where $p\geq2$, $\beta >0$,
and $\dd(x,e)$ is the \textit{Carnot-Carath\'eodory distance}
of the point $x \in \He_1$ from the identity element $e$ of $\He_1$,
all satisfy a log-Sobolev inequality; this was shown by Hebisch and Zegarlinski in \cite{H-Z}.  For the inequality on Lie groups see  also Feng and Li \cite{Fe-Li1}, Gordina and Luo \cite{Gor-Luo2022},  Chatzakou,  Kassymov and  Ruzhansky  \cite{Ch-Ka-Ru15652}, Suguro \cite{Sug24} and Bonnefont,    Chafaï and  Herry \cite{Bo-Ch-He2020}. Entropy dissipation  in the Heisenberg group has been studied by  Feng and   Li in \cite{Fe-Li}.

We briefly recall the   notion of the Carnot-Carath\'eodory metric on $\He$.

A Lipschitz curve
$\gamma:[0,1]\to\He$ is said to be \textit{admissible} if $\gamma'(s) =
a_1(s)X_1(\gamma(s)) + a_2(s)X_2(\gamma(s))$, a.e., for 
given measurable functions $a_1(s)$, $a_2(s)$,
and has length
$
l(\gamma) = \int_0^1\left(a_1^2(s) + a_2^2(s)\right)^{1/2}ds.
$
The Carnot-Carath\'eodory metric is then defined by
$$
\dd(x,y) := \inf\{l(\gamma):\gamma\ \textrm{is an admissible path joining}\ x\
\textrm{and}\ y\}.
$$

We also have that $x=(x_1, x_2, x_3)\mapsto \dd(x,e)$ is smooth 
for $(x_1, x_2) \neq 0$, but has singularities at points of the form
$(0,0, x_3)$. 
 Thus, the unit ball in the metric above has singularities on the $x_3$-axis.
In our analysis, we will  use the following result about the
Carnot-Carath\'eodory distance (see, for example, \cite{H-Z}, \cite{Mo}).

\begin{prop}
\label{eik}
Let $\nabla$ be the  gradient and $\Delta$ be the sub-Laplacian on $\He_1$.
Then $ \|\nabla \dd(x,e) \|=1$ for all $x=(x_1, x_2, x_3)\in\He$ such that
$(x_1, x_2)\neq0$. Also there exists a positive constant $K$ such  that $\Delta
\dd(x,e)<K/\dd(x,e)$ in the sense of distributions.
\end{prop}

\section{Assumptions and main results }

In this section we present the  hypothesis and the statement of the main result.  Without loss of generality, assume the single-site space to be the origin $0 \in \Z^d$. Let $\Spin$ be the corresponding spin space. To ease the notation,
we denote the Hamiltonian by
\[
H(x) = \phi(x) + \sum_{j=-d}^{d}J_{j} V_j(x), \quad x\in \Spin,
\]
where $e_j \in \Z^d$ is the vector with components
$e_{j,i} = \1_{i=j}$ and $V_{\pm j}(x) := V(x, \omega_{\pm e_j})$,
$j=1,\ldots,d$. In other words, we freeze the boundary conditions
$\omega_{-e_d}, \ldots, \omega_{e_d}$ at the $2d$ neighbors 
$\pm e_1, \ldots, \pm e_d$ of the origin.
Of course, we need to assume that the functions $\phi$ and $V_j$
are such that $\int_\Spin \exp(-H(x)) dx < \infty$ so that
the measure with density $\exp(-H(x))$ be normalizable to
a probability measure which (again suppressing the $\omega$) we
simply denote as $\E$:
\[
\E(dx) = Z^{-1} e^{-H(x)} dx.
\]
Before stating the main results, we introduce a number of natural
hypotheses. 

~

\textbf{The main assumption}$~$

 The single site measure without interactions  (consisting only of the phase) 
$$\mu(dx)=\frac{e^{-\phi(x)}dx}{\int e^{-\phi(x)} dx}$$  
 satisfies
the log-Sobolev inequality with a constant $c$.

~

\textbf{Assumptions on the  phase and the interaction potential}$~$

We also assume that   $J_j>0$ and  that $\phi$ and the $V_j$ are non negative twice  continuously
differentiable satisfying the following ``geometric'' conditions:
there exists a nonnegative 
$C^2$ function $\phi_1$ 
such that
\begin{equation}
\label{H2.1}
\nabla \phi = \phi_1 \nabla \dd .
\end{equation}
Similarly, for each $V_j$:
\begin{equation}
\label{H2.2}
\nabla V_j=  U_j \nabla \dd,\end{equation}
where $U_j$ are nonnegative $C^2$ functions.
 The gradient vector $\nabla \dd$ is uniformly bounded in magnitude from above and below:
there exist constants $\tau$ and $\xi$ such that, for all
$x \in \Spin$,
\begin{equation}
\label{H2.3}
\xi \le \|\nabla \dd\| \le \tau .
\end{equation}
Instead of speaking of a metric $\dd$, we shall, for the purposes of
this section, speak of positive functions $\dd$,
such that there exists a constant $\theta$ with
\begin{equation}
\label{H2.4}
|\Delta \dd| \le \frac{\theta}{\dd},
\end{equation}
for all $j$ and all $x$.
Moreover, we require that there exists $k_0 >0$ and 
$ p \ge 2$ 
such that
\begin{equation}
\label{H2.5}
  k_0 \phi\le d\phi_1 \ \text{and} \ d^p\leq \phi  
\end{equation}
and
\begin{equation}
\label{H2.6}
k_0V_j\leq \dd U_{j},
\end{equation}
for all $j$ and $x$.
 Furthermore,  we assume 
 \begin{equation}
\label{Hinfty}
 V_j \rightarrow +\infty \text{ \ as  \ }  \dd (\omega_{ e_j})\rightarrow +\infty
\end{equation}
and that  
 $\exists ~ s \le p$  and $k>0$ such that
\begin{equation}
\label{H2.8}
\| \nabla V_j \|^2 \le k + k \dd^s + k \dd^s(\omega_{ e_j}),
\end{equation}
\begin{equation}
\label{H2.9}
 V_j \le k + k \dd^s + k \dd^s(\omega_{ e_j}).
\end{equation}
Three last  assumptions follow.  These, as shown in section \ref{sectExample1},  are natural assumptions that are easily verified for Hamiltonians that are given as functions of $\dd$.  For any $x,y\in \Spin$ we assume that there exists a $\lambda>1$ such that
\begin{equation}
\label{lowerH}
H(x \cdot y)  \le  \lambda H(x)+\lambda H (y)
\end{equation}
where $\cdot$ the group operation,  while for $x^{-1}$ the inverse of $x$ in respect to the group operation,
\begin{equation}
\label{Hinverse}
H(x^{-1})= H (x)
.\end{equation}
If we consider $\gamma:[0,t]\rightarrow \Spin $ a geodesic from $0$ to $x \in \Spin$ then
\begin{equation}
\label{Hgeo}
H(\gamma (s))\leq H (x)
\end{equation} for every $s\in [0,t]$.

We can now state the main theorem related to the general framework.

\begin{thm}
\label{thmGENERAL} 
Let $f\colon \mathbb{\Spin}^{\mathbb{Z}^d} \to \mathbb{R}$. 
If assumptions (\ref{H2.1})-(\ref{Hgeo}) hold,
if the single-site measure $\mu$ satisfies
a log-Sobolev inequality, 
 then the Gibbs measure
$\nu$ satisfies a log-Sobolev inequality:
\[
\nu f^{2}\log\frac{f^{2}}{\nu
f^{2}}\leq \mathfrak{C} \ \nu \left\| \nabla f
\right\|^2,
\]
for some positive constant $\mathfrak{C}$. 
\end{thm}

Assumptions (\ref{H2.3})-(\ref{H2.4}) refer to the distance, while assumptions (\ref{H2.1})-(\ref{H2.2}) on how the phase and the interactions are formed from it. Conditions (\ref{H2.5})-(\ref{H2.9}) consist of the main assumptions about the phase and the interaction and how these relate with each other. (\ref{H2.5}) and (\ref{Hinfty}) imply that both phase and interactions are unbounded functions, with the first one increasing faster than a quadratic. 
 The main
assumption about the the relation between interactions and phase, (\ref{H2.5}) and (\ref{H2.8}), 
is that the phase $\phi(x) $ dominates over
the interactions, in the sense that
$$\left\| \nabla V(x_{i},\omega_{j})\right\|^{2}\leq k+k(
d^s(x_{i})+d^s(\omega_j))\leq k+k( \phi(x_i
)+ \phi(\omega_j ))$$
for $s\leq p$.  This allows to consider interactions of higher order than the quadratic in \cite{Y} and \cite{B-H}.  
Furthermore, the main   assumption about the   phase $\phi$ is  that the  single site
measure $\mu$ satisfies the log-Sobolev
inequality. This last differentiates the current result from the one in    \cite{Pa1} where higher order interactions where also considered, but a stronger assumption   was adopted, namely that the one site measure $\E^{i,\omega}$ satisfies a log-Sobolev inequality uniformly on the boundary conditions $\omega$.  
The fact that the main assumption about the log-Sobolev inequality on a site is now relaxed and does not involve interactions is very important since that allows to consider perturbations of a convex $\phi$ analogue  to the ones considered in \cite{H-Z} and \cite{Pa5}, as long as the phase still dominates on the interactions as required from (\ref{H2.5}), (\ref{H2.8}) and (\ref{H2.9}). For instance, one can consider $$\phi(x)=d^{p}(x)+d^{p-1}(x)cos(d(x))$$
as presented for the case of the Heisenberg group in example (\ref{HGexample}). Accordingly, the conditions on the current paper are weaker than the ones in \cite{Y} and \cite{B-H} where a phase convex at infinity was considers, while the interactions had to be at most of quadratic order.

We briefly mention some consequences of this result. 
\begin{cor}
Let $\nu$ be as in Theorem \ref{thmGENERAL}. Then $\nu$ satisfies the spectral gap
inequality
$$
\nu\left ( f-\nu f\right )^2 \leq \mathfrak{C} \nu \left\| \nabla f
\right\|^2
$$
where $\mathfrak{C} $ is as in Theorem \ref{thmGENERAL}.
\end{cor}
The proofs of the  next two can be found in [B-Z].
\begin{cor}
Let $\nu$ be as in Theorem \ref{thmGENERAL} and suppose $f:\Omega \to \R$ is such
that $ \| \nabla f \| ^2 \infty < 1$. Then
$$
\nu\left( e^{\lambda f}\right) \leq \exp\left\{\lambda \nu(f) + \mathfrak{C}
\lambda^2\right\}
$$
for all $\lambda>0$ where $\mathfrak{C} $ is as in Theorem \ref{thmGENERAL}.
Moreover, the following 'decay of tails' estimate holds true
$$
\nu\left\{\left \vert f - \int fd\nu\right \vert \geq h\right\} \leq
2\exp\left\{-\frac{1}{\mathfrak{C} }h^2\right\}
$$
for all $h>0$.\end{cor}
\begin{cor}
Suppose that our configuration space is actually finite dimensional, so that we
replace $\Z^d$ by some finite graph $G$, and $\Omega = (\Spin)^G$. Then
Theorem \ref{thmGENERAL} still holds, and implies that if $\mathcal{L}$ is a
Dirichlet operator satisfying
$$
\nu\left(f\mathcal{L} f\right) = -\nu\left( \vert \nabla f\vert ^2\right),
$$
then the associated semigroup $P_t = e^{t\mathcal {L}}$ is ultracontractive.
\end{cor}
For further properties of Markov semigroups on the lattice see Wong \cite{Won2021inf}.
Next, we present an example of a measure that satisfies the
hypothesis of  Theorem \ref{thmGENERAL}.

\subsection{The Case of Heisenberg Group}\label{sectExample}

 As an  example of a measure  $\E^{i,\omega}$ that satisfies
 the conditions of Theorem   \ref{thmGENERAL}
 one can consider the following  measure on the Heisenberg group 
$$\mathbb{E}^{\Lambda,\omega}(dX_\Lambda) =
\frac{e^{-H^{\Lambda,\omega}}dX_\Lambda} {Z^{\Lambda,\omega}}$$ 
  where for any $\Lambda\Subset \Z^d,\omega \in 
\Omega$ the  Hamiltonian is defined as 
\begin{equation}
H^{\Lambda,\omega}(x_\Lambda) =\\ \sum_{i\in \Lambda}\left(\dd^{p}(x_i) +c\dd^{p-1}(x_i)cos(\dd(x_i))\right)+
\delta\sum_{i\in \Lambda ,j\sim i}(\dd(x_i)+\dd(\omega_j))^{r}
\label{HGexample}\end{equation}
for $c\in \mathbb{R}$, $ \delta>0$ and $p,r\in \mathbb{N}$ s.t. $\frac{p+2}{2}\geq r>2$, where $\dd$ the
Carnot-Carath\'eodory  distance. It should be noted that the main assumption that the single site measure without interactions  (consisting only of the phase) satisfies the log-Sobolev inequality follows from \cite{H-Z}.

Then the main result related to the infinite
volume Gibbs measure associated with this local specification follows:
 
\begin{thm}\label{thmExample}Consider $ \He$ the Heisenberg group and let
$f\colon \mathbb{ \He}^{\mathbb{Z}^d} \to \mathbb{R}$. If
$\{\mathbb{E}^{\Lambda,\omega}\}_{\Lambda\Subset \Z^d,\omega \in
\Omega}$ as in (\ref{HGexample}). Then the infinite-dimensional Gibbs measure
$\nu$ for the local specification
$\{\mathbb{E}^{\Lambda,\omega}\}_{\Lambda\Subset \Z^d,\omega \in
\Omega}$ satisfies the log-Sobolev  inequality
$$\nu f^{2}\log\frac{f^{2}}{\nu f^{2}}\leq \mathfrak{C} \ \nu \left\vert \nabla
f
\right\vert^2$$                              
for some positive constant $\mathfrak{C}$. \end{thm}

A few words about the structure of the paper. In section \ref{section3} we show a coercive inequality as well as the Poincare inequality for the one site measure $\E^i$. In the next section we present the first  sweeping out inequalities and show convergence to equilibrium, while in section \ref{secLS} a weak logarithmic Sobolev inequality is obtained for the  one site measure $\E^i$.  Further sweeping out inequalities are obtained in section \ref{section6} together with a log-Sobolev inequality for the product measure. In the next section \ref{PROOF} we gather all the previous bits together to prove  the main result of Theorem \ref{thmGENERAL}. Finally, in section \ref{sectExample1} we present the proof of Theorem \ref{thmExample}.

\section{A coercive inequality for the single-site space} \label{section3}
In this section we present a single-site coercive inequality that will provide the main tool in order to control the higher interactions. This  coercive inequality is on the line of the U-bound inequalities presented in \cite{H-Z} in order to prove log-Sobolev inequalities on a typical analytic framework (see also  \cite{INGLIS201176}, \cite{Dag2022}  and  \cite{Da-Ze2021}).  Furthermore, as we show in Lemma \ref{spectralgap1d} this coercive inequality  will imply  the spectral gap inequality for $\E^{i,\omega}$ uniformly on $\omega$. In a recent work, weak U-bound inequalities have been used to prove Spectral Gap inequalities (see \cite{Da-Qi-Ze22}) and Super-Poincare inequalities (see \cite{Qiu24}).
 \begin{lem}\label{coers}
Under   assumptions (\ref{H2.1})-(\ref{Hgeo}), there exists $C_{0}>0$ such that,
for all $r \le p$,
\[
\E \dd^r f^2 \le C_{0} \E|\nabla f|^2 + C_{0} \E f^2,
\]
and
\[\E Hf^{2}\leq C_{0} \E  |\nabla f|^{2}+C_{0} \E f^{2}\]
for any smooth function $f$ with compact support.
\end{lem}
\begin{proof}
It is clear that it suffices to prove the inequality for $r=2(p-1)$.
Indeed, if $\E  \dd^{2(p-1)} f^2 \le C \E |\nabla f|^2 + C \E f^2$
holds then for all $r \le 2(p-1)$ we have
$\E \dd^r f^2 = \E [\dd^r f^2; d \le 1] + \E [\dd^r f^2; d > 1]
\le \E f^2 + \E \dd^{2(p-1)} f^2 \le C \E |\nabla f|^2 + (C+1) \E f^2$.
By homogeneity, in all calculations, we will forget the normalizing constant
$Z$ and think of $\E(dx)$ as being equal to $e^{-H(x)} dx$.
In other words, we may, without loss of generality, assume that $Z=1$.
Let $f$ be a smooth function with compact support and write
\[
\E |\nabla f|^2 = \int |\nabla f|^2 e^{-H} dx.
\]
Since 
\[
\nabla (f e^{-H}) = (\nabla f) e^{-H} 
-  (\nabla H) e^{-H} f,
\]
upon taking the inner product   with $\dd\nabla \dd$ on both sides we get
\[
 \dd\langle \nabla \dd,\nabla H\rangle e^{-H } f=\dd\langle\nabla \dd,\nabla f\rangle e^{-H}-\dd\langle \nabla \dd, \nabla (f e^{-H})\rangle .
\]
Hence,
\begin{align*} \underbrace{
\E \dd\langle \nabla\dd,\nabla H\rangle f}_{\I_{1}}=&\E\dd\langle\nabla\dd,\nabla f\rangle-\int \dd\langle \nabla\dd, \nabla (f e^{-H})\rangle dx  \\ \leq&
  \E\dd|\nabla\dd||\nabla f|-\int \dd\langle \nabla\dd, \nabla (f e^{-H})\rangle dx
 \\ \leq& \tau
  \E\dd |\nabla f|- \underbrace{\int \dd\langle \nabla\dd, \nabla (f e^{-H})\rangle dx}_{\I_2}
\end{align*}
where above we used  (\ref{H2.3}). Let $X$ be any of the H\"ormander generators of $\Spin$. 
Then, by the structural assumption,
we have the integration-by-parts
formula
\[
\int F (XG) dx = - \int (XF) G dx
\]
for smooth functions $F$ and $G$ with  compact support. As a consequence,
the integration-by-parts formula
\[
\int f \langle \nabla \Phi, \nabla \Psi \rangle dx
= - \int \langle \nabla \Phi, \nabla f \rangle \Psi dx
- \int (\Delta\Phi) \Psi f dx,
\]
holds, and so
\[
\I_2 
=\int \dd\langle \nabla\dd, \nabla (f e^{-H})\rangle dx=-\int \dd|\nabla\dd| f e^{-H} dx-\int \dd(\Delta\dd)f e^{-H} dx
\geq- \tau\E\dd f-\theta \E f
\]
because of  (\ref{H2.3}) and (\ref{H2.4}). Since $H = \phi+\sum_j J_{j}V_j$, the  first term is
\begin{align*}
\I_1 
& = \E\dd\langle \nabla\dd,\nabla H\rangle f=\E\dd\langle \nabla\dd,\nabla\phi\rangle f+\sum_jJ_{j}\E\dd\langle \nabla\dd,\nabla V_j\rangle f= 
\\
&= \E \dd\phi_1 \vert \nabla\dd\vert f+\sum_jJ_{j}\E\dd U_{j}\vert \nabla\dd\vert f \geq \xi\ k_{0}\E\phi f+\xi k_{0}\sum_jJ_{j}\E V_{j}  f
\end{align*}
where above we used at first (\ref{H2.1})-(\ref{H2.2}) and in the last inequality (\ref{H2.3}), (\ref{H2.5}) and (\ref{H2.6}). Combining all that  we arrive at
\[
\E\phi f+\sum_jJ_{j}\E  V_{j}  f\leq\ \frac{1}{\xi\ k_{0}}
  (\tau \E \dd |\nabla f|+\tau \E\dd f+\theta \E f) . \]
 If we replace $f$ by $f^2$  and use   Cauchy-Schwarz  inequality we obtain
\begin{align*}
\E Hf^{2}\leq\  & 
  \frac{1}{ \xi\ k_{0}}(2\tau\E\dd f |\nabla f|+ \tau \E\dd f^{2}+\theta \E f^{2} )
  \leq\frac{1}{\xi\ k_{0}}(\tau\E  |\nabla f|^{2}+\tau\E\dd^{2}f ^{2}+\tau \E\dd f^{2}+\theta \E f^{2})  \\  = & \frac{1}{ \xi\ k_{0}}\{\tau\E  |\nabla f|^{2}+\tau\E(\mathrm{I}_{\{\frac{4\tau}{\xi\ k_{0}}\leq\dd^{p-2}\}} +\mathrm{I}_{_{\{\frac{4\tau}{\xi\ k_{0}}>\dd^{p-2}\}}})\dd^{2}f ^{2}+\\  & +\tau\E(\mathrm{I}_{\{\frac{4\tau}{\xi\ k_{0}}\leq \dd^{p-1}\}}+\mathrm{I}_{\{\frac{4\tau}{\xi\ k_{0}}>\dd^{p-1}\}})\dd f^{2}+ \theta \E f^{2} \} \\ 
  \leq &\frac{1}{\xi\ k_{0}}\left\{\tau \E  |\nabla f|^{2}+\frac{\xi\ k_{0}}{2} \E \dd^{p}f ^{2}+ \left(4\tau(\frac{4\tau}{\xi\ k_{0}})^\frac{2}{p-2}+4\tau(\frac{4\tau}{\xi\ k_{0}})^\frac{1}{p-1}+ \theta\right) \E f^{2}  \right\} \\  
  \leq &\frac{1}{2} \E Hf ^{2}+
  \frac{1}{ \xi\ k_{0}}\left\{\tau\E  |\nabla f|^{2}+ \left(4\tau(\frac{4\tau}{\xi\ k_{0}})^\frac{2}{p-2}+4\tau(\frac{4\tau}{\xi\ k_{0}})^\frac{1}{p-1}+ \theta\right) \E f^{2}\right\}
\end{align*}
since  $d^p\leq \phi \leq H$, because of
   (\ref{H2.5}) and the non negativity of $\phi$ and $V_j$. Again, for the same reason we obtain
\begin{align*}
\E\dd^{p}  f^{2}\leq\E Hf^{2}\leq\frac{2}{ \xi\ k_{0}}\left\{\tau\E  |\nabla f|^{2}+ \left(4\tau(\frac{4\tau}{\xi\ k_{0}})^\frac{2}{p-2}+4\tau(\frac{4\tau}{\xi\ k_{0}})^\frac{1}{p-1}+ \theta\right) \E f^{2}\right\}
\end{align*}
which proves  the  inequality.
\end{proof}
 We will now prove the  Poincare inequality for the one site measure $\E^i$ for a constant uniformly on the boundary conditions. For the Poincare inequality on Carnot Group see Dagher \cite{Dag24},  Chatzakou, Federico and Zegarlinski (\cite{Ch-Fe-Ze1} and \cite{Ch-Fe-Ze2}), 
 Dragoni,  Kontis and Zegarliński \cite{Dr-Ko-Ze12},  Dagher, Qiu, Zegarlinski and Mengchun Zhang  \cite{Da-Qi-Ze22}. The proof follows closely  the proof of   the local Poincar\'e inequalities from Saloff-Coste  \cite{SC} and   Varopoulos and  Coulhon \cite{V-SC-C}.
\begin{lem}\label{spectralgap1d}Under   assumptions (\ref{H2.1})-(\ref{Hgeo}),  $\E^{i,\omega}$ satisfies the spectral gap inequality
\[\E^{i}(f-\E^if)^2\leq c_p\E^i \|\nabla_if\|^2\]
for some constant $c_p>0$ uniformly on the boundary conditions.
\end{lem}
\begin{proof}
We denote  set $V(R)=\{x_i:H ^i\leq R\}$. Then,  if we  define  $a(f)=\frac{1}{\left\vert V_{R} \right\vert}\int_{V_{R}}f(x_i)dx_i$, where $\left\vert V_R\right\vert= \int_{V_{R}}dx_i$,  we can  compute  
\[\E^i(f-a(f))^2=\underbrace{\E^i(f-a(f))^2\mathrm{I}_{V(R)}}_{\II_1}+\underbrace{\E^i(f-a(f))^2\mathrm{I}_{V(R)^c}}_{\II_2}\]
where $V(R)^c$ the complement of $V(R)$. Since $\phi, V$ and $J_{ij}$ are all no negative, $H^i\geq0$ and so the first term is 
\[\II_1\leq \frac{1}{Z^i}\int_{V(R)}( f(x_i)-a(f))^2dx_{i}.\]
If we now  use the invariance  of the  measure $dx_i$ with  respect to the group operation we can write $a(f)=\frac{1}{\left\vert V(R) \right\vert}\int f(x_iz_i)\mathrm{I}_{V(R)}(x_{i}z_{i})dz_i$. If we substitute this expression on the last inequality and use Cauchy-Schwarz inequality we obtain 

\[\II_1\leq \frac{1}{Z^i}\frac{1}{\vert V(R) \vert}\int( f(x_i)-f(x_iz_i))^2\mathrm{I}_{V(R)}(x_{i}z_{i})\mathrm{I}_{V(R)}(x_{i} )dz_{i}dx_{i}\]
where above we also considered $R$ large enough so that $\vert V(R)\vert>1$, i.e. $\frac{1}{\vert V(R) \vert^2}\leq\frac{1}{\vert V(R) \vert}$. 
Consider  a geodesic  $\gamma:[0,t]\rightarrow H $  from $0$ to $z_i$, such that $\vert \dot \gamma(t) \vert \leq 1$. Then      we can write 
\begin{align*}( f(x_{i})- f(x_iz_{i}))^2 =&\left(\int_{0}^{\dd(z_i)} \frac{d}{ds} f(x_{i}\gamma(s))ds\right)^2=\left(\int_{0}^{\dd(z_i)} \nabla_{i}\ f(x_{i}\gamma(s))\cdot \dot\gamma(s)ds\right)^2  \\ \leq &\dd(z_i)\int_{0}^{\dd(z_i)}\| \nabla_{i}\ f(x_{i}\gamma(s))\|^2 ds . \end{align*}
From  the last inequality, we can  bound
\[\II_1\leq\frac{1}{Z^i}\frac{1}{\vert V(R)\vert}\int \dd(z_i)\int_{0}^{\dd(z_i)}\| \nabla_{i}\ f(x_{i}\gamma(s))\|^2 ds\mathrm{I}_{V(R)}(x_{i}z_{i})\mathrm{I}_{V(R)}(x_{i} )dz_{i}dx_{i}.\]
We observe  that for   any $x_i  \in V(R)$ and $ x_i z_i\in V(R)$ we obtain 
\[H^{i}(z_i)=H^{i}(x_i^{-1}x_iz_i)\leq  \lambda H^{i}(x_i^{-1} )+\lambda H^{i}( x_iz_i) \leq \lambda +\lambda H^{i}(x_i  )+H^{i}( x_iz_i)\leq  \underbrace{2R\lambda }_{:=r_1}\] 
because of (\ref{lowerH}) and (\ref{Hinverse}). Furthermore, using (\ref{lowerH}) and (\ref{Hgeo})   we can calculate
\[H^{i}( x_i \gamma(s))\leq   \lambda H^{i}(x_{i})+\lambda H^{i}(\gamma(s))\leq   \lambda H^{i}(x_{i})+\lambda H^{i}(z_i)< \underbrace{(2\lambda^2 +\lambda )R}_{:=r_2}.\]
From  (\ref{H2.5}), since $\dd(z_i)\leq \phi(z_i)^\frac{1}{p}$   we can also bound
\[\dd(z_i)\leq  \phi(z_i)^\frac{1}{p}\leq H^i(z_i)^\frac{1}{p}\leq r_1^{1/p}.\]
So, we get
\begin{align*} \II_1&\leq\  \frac{ r_1^{1/p}}{Z^i\vert V(R) \vert} \int \int \int_{0}^{\dd(z_i)}\| \nabla_{i}\ f(x_{i}\gamma(s))\|^2  \mathbb{\mathcal{I}}_{V(r_2)}( x_i \gamma(s))\mathbb{\mathcal{I}}_{V(r_1)}(z_i )dsdz_idx_{i}.
\end{align*}
Using again the  invariance   of  the measure we can write
\begin{align*} \II_1&\leq  \frac{ r_1^{1/p}}{Z^i\vert V(R)\vert} \int \int \int_{0}^{\dd(z_i)}\| \nabla_{i}\ f(x_{i})\|^2   \mathbb{\mathcal{I}}_{V(r_2)(x_{i})}\mathbb{\mathcal{I}}_{V(r_1)}(z_i )dsdx_{i}dz_i\\  &  =\frac{ r_1^{1/p}}{Z^i\vert V(R)\vert} \int \int  \dd(z_{i})\| \nabla_{i}\ f(x_{i})\|^2  \mathbb{\mathcal{I}}_{V(r_2)}(x_{i})\mathbb{\mathcal{I}}_{V(r_1)}(z_i )dx_{i}dz_i .\end{align*}
Notice that for    $z_i  \in V(r_1)$  one can bound as before  $  \dd(z_i)\leq  \phi(z_i)^\frac{1}{p}\leq H^i(z_i)^\frac{1}{p}\leq r_1^{1/p}$, and so
\begin{align*}\II_1&\leq  \frac{ r_1^{2/p}}{Z^i\vert V(R)\vert} \int \int   \left(\| \nabla_{i}\ f(x_{i})\|^2  \mathbb{\mathcal{I}}_{V(r_2)}(x_{i})\right)dx_{i}\mathbb{\mathcal{I}}_{V(r_1)}(z_i )dz_i \\
 & \leq \frac{ r_1^{2/p} \vert V(r_1)\vert}{ \vert V(R)\vert Z^{i,\omega}} \int \| \nabla_{i}\ f(x_{i})\|^2   \mathbb{\mathcal{I}}_{V(r_2)}(x_{i})dx_{i} .
\end{align*}
Since, for $x_i \in V(r_2)$, we    have $e^{-H^{i,\omega}}\geq e^{-r_2 } $, the last quantity can be   bounded by 
\[ \II_1 \leq  \frac{ e^{r_2}r_1^{2/p} \vert V(r_1)\vert}{ \vert V(R)\vert}\mathbb{E}^{i}\| \nabla_{i}\ f\|^2   .
\]
If now we take under account that $\frac{ \vert V(r_1)\vert}{ \vert V(R)\vert }=\frac{ \vert V(2R\lambda)\vert}{ \vert V(R)\vert }\geq 1$, as well as that because of (\ref{Hinfty}), the limit $\frac{ \vert V(2\lambda R)\vert}{ \vert V(R)\vert }\rightarrow1$ as $\sum_{j\sim i}\dd(\omega_j)\rightarrow \infty$,
we then  observe that  $\frac{ \vert  V( r_1)\vert }{ \vert V(R)\vert }$ is bounded from above uniformly on $\omega$ from a constant. Thus, we finally obtain that
\begin{align*} \II_1&\leq  C(R) \mathbb{E}^{i}\| \nabla_{i}\ f(x_{i})\|^2   \end{align*}
for some positive  constant $C(R)$.

We will now compute $\II_2$. We have 
\[\II_2\leq\E^i(f-a(f))^2\frac{H^{i}}{R}\leq\frac{C_0}{R}\E^i |\nabla f|^{2}+\frac{C_0}{R} \E ^{i}(f-a(f))^2\]
where above we used  Lemma \ref{coers}.
 Combining  all the above we obtain
 \[\E ^{i}(f-a(f))^2\leq(C(R)+\frac{C_0}{R})\E^i \| \nabla f\|^{2}+\frac{C_0}{R} \E ^{i}(f-a(f))^2\]
For $R$ large   enough so that $\frac{C_0}{R}<1$ we get
\[\E ^{i}(f-a(f))^2\leq\frac{C(R)+\frac{C_0}{R}}{1-\frac{C_0}{R} }\E^i \| \nabla f\|^{2}.\]
Since $\E ^{i}(f-\E^i(f))^2\leq4\E ^{i}(f-k)^2 $ for any real number k, the result follows.
\end{proof}

\section{Sweeping out inequalities and  convergence to the Gibbs measure} \label{section4}

In this section we prove sweeping out inequalities that interchange the positions between the gradient and the measure. Sweeping out inequalities are a weak version of the sweeping relations  introduced by Zegarlinski to prove log-Sobolev inequalities for continuous spin systems on the lattice (see for instance \cite{Zega1990} and \cite{G-Z}).

Recall  the  definition of the  operator $\E^\Lambda$
and the definition of $\nabla_j f$ as being the gradient of a 
function $f:\Spin^{\Z^d} \to \R$ with respect to the coordinate
$\omega_j \in \Spin^{\{j\}}$.
Also, recall the assumption that there is a unique Gibbs measure $\nu$.
By our notational conventions, for $i \in \Z^d$, the quantity $\E^i f$
is a function on $\Z^d$ that depends only on the $2d$ variables
$x_j$, with $j\in \Z^d$ ranging over the neighbors of $i$ and the $x_j$'s that comprise the input of   $f$ excluding $x_i$.
Fixing a neighbor $j$, the 
gradient $\nabla_j \E^i f$ is then gradient with respect
to $x_j$. Denoting by $X_j^1, \ldots, X_j^n$ the H\"ormander
system for $\Spin^{\{j\}}$, we have
$\nabla_j ( f) = (X_j^1 f, \ldots, X_j^n f)$,
so that $\|\nabla_j (\E^i f)\|^2 = \sum_{\alpha=1}^n (X^\alpha_j \E^i)^2 $.
We have

\begin{lem}\label{sweep1}
Suppose that (\ref{H2.1})-(\ref{Hgeo}) hold.  Let $i,j \in \Z^d$ be neighbors. 
Then, for $J$ sufficiently small, there are constants $D_1>0$ and $0<D_2< 1$ such that
\[
\nu \| \nabla_j (\E^i f)\|^2 \le D_1 \nu \| \nabla_j f\|^2
+ D_2 \nu \| \nabla_i f\|^2 .
\]
\end{lem}
\begin{proof}
Fix $i \in \Z^d$ and let $j$ be one of its neighbors.
We compute $(X^\alpha_j (\E^i f))^2$. Letting $\rho_i$ be the 
density of $\E^i$ with respect to $dx_i$, we have,
using Leibniz' rule and $(a+b)^2 \le 2 a^2+2 b^2$,
\begin{align}
(X^\alpha_j (\E^i f))^2 
&=\bigg( 
\int \rho_i (X^\alpha_j f) dx_i
+ \int (X^\alpha_j \rho_i) f dx_i  
\bigg)^2
\le 
2 \E^i (X^\alpha_j f)^2
+ 2\bigg( \int (X^\alpha_j \rho_i) f dx_i\bigg)^2,
\label{E4.1}
\end{align}
where we used Jensen's inequality to pass in the square inside the
expectation in the first term.
If we sum over $\alpha$ and integrate over $\nu$, the first term
on the right becomes $\nu \|\nabla_j f\|^2$, which is what we need.
For the second term, we need to take into account the specific form of
the density $\rho_i = e^{-H^i}/Z^i$. Note that $H^i$ depends on $x_i$
and the variables $x_\ell$, where $\ell$ ranges over the neighbors
of $i$, including $j$, but $Z^i$ does not depend on $x_i$. 
Taking this into account and using Leibniz' rule again, 
we easily arrive at
\begin{equation}
\label{E4.2}
\int (X^\alpha_j \rho_i) f dx_i
= - \E^i[ f(X^\alpha_j H^i - \E^i (X^\alpha_j H^i))]
= - \E^i[(f-\E^i f)\, (X^\alpha_j H^i)].
\end{equation}

The computation for the last one is as follows:
$X_j \rho_i = (X_j e^{-H^i})/Z^i - e^{-H^i}(X_j Z^i)/ (Z^i)^2$.
But $X_j e^{-H^i} = - e^{-H^i} (X_j H^i)$, and
$$X_j Z^i = X_j \int e^{-H^i} dx_i = - \int e^{-H^i}(X_j H^i) dx_i.$$
So $X_j \rho_i = -(e^{-H^i}/Z^i) (X_j H^i)
+ (e^{-H^i}/Z^i) \int (e^{-H^i}/Z^i) (X_j H^i) dx_i 
= -\rho_i (X_j H^i) + \rho_i \int \rho_i (X_j H^i) dx_i$

At this point, we use Jensen's inequality again,
\begin{equation}
\label{E4.3}
\left( \int (X^\alpha_j \rho_i) f dx_i\right)^2
\le \E^i[(f-\E^i f)^2\, (X^\alpha_j H^i)^2],
\end{equation}
and then take into account the specific form of $H^i$.
Since the differential operator $X_j^\alpha$ acts on $x_j$,
only the one of the interactions terms survives, giving
\begin{equation}
\label{E4.4}
X^\alpha_j H^i = J_{ij} X^\alpha_j V(x_i,x_j).
\end{equation}
Therefore, using (\ref{H2.8})
\begin{align*}
\sum_{\alpha=1}^n \left( \int (X^\alpha_j \rho_i) f dx_i\right)^2
&\le J_{ij}^2 \, \E^i \big[ (f-\E^i f)^2 \|\nabla_j V(x_i,x_j)\|^2\big]
\\
&\le k J_{ij}^2 \E^i (f-\E^if)^2 
+  k J_{ij}^2 \E^i (f-\E^if)^2 \dd(x_i)^s
+ k J_{ij}^2 \E^i (f-\E^if)^2 \dd(x_j)^s.
\end{align*}
Summing up the first display of this proof over $\alpha$
and integrating over $\nu$ we obtain 
\begin{multline}
\label{E4.5}
\nu \|\nabla_j(\E^i f)\|^2 \le 2 \nu \|\nabla_j f\|^2
+ 2k J_{ij}^2 \nu[ (f-\E^if)^2 ]
\\
+  2k J_{ij}^2 \nu[ (f-\E^if)^2 \dd(x_i)^s]
+ 2k J_{ij}^2 \nu[ (f-\E^if)^2 \dd(x_j)^s].
\end{multline}
From the single-site  coercive inequality of Lemma \ref{coers},
\begin{multline}
\nu[ (f-\E^if)^2 \dd(x_i)^s] 
= \nu \E^i [ (f-\E^if)^2 \dd(x_i)^s]
\le C_0\, \nu \|\nabla_i f\|^2 + C_0\, \nu [(f-\E^if)^2]
\label{E4.6}
\end{multline}
and
\begin{multline}
\nu[ (f-\E^if)^2 \dd(x_j)^s] 
= \nu \E^j [ (f-\E^if)^2 \dd(x_j)^s]
\le C_0\, \nu \|\nabla_j (f-\E^i f)\|^2 + C_0\, \nu [(f-\E^if)^2],
\\ 
\le 2C_0\, \nu \|\nabla_j f\|^2
+ 2C_0\, \nu \|\nabla_j(\E^i f)\|^2
+ C_0\, \nu [(f-\E^if)^2].
\label{E4.7}
\end{multline}
Substituting these last two into \eqref{E4.5} gives
\begin{multline*} 
\nu \|\nabla_j(\E^i f)\|^2  
\le (2+4kJ_{ij}^2 C_0)\nu\|\nabla_j f\|^2+  2kJ_{ij}^2 (1+2C_0) \nu [(f-\E^if)^2] 
+2kJ_{ij}^2 C_0\, \nu\|\nabla_i f\|^2
\\
+ 4kJ_{ij}^2 C_0\, \nu \|\nabla_j(\E^i f)\|^2.
\end{multline*}
We can now use the Poincare inequality from Lemma \ref{spectralgap1d} to bound the variance
\begin{multline*} 
\nu \|\nabla_j(\E^i f)\|^2  
\le (2+4kJ_{ij}^2 C_0)\nu\|\nabla_j f\|^2+  2kJ_{ij}^2 (c_p+C_0+2c_pC_0) \nu \|\nabla_i f\|^2+
\\
+ 4kJ_{ij}^2 C_0\, \nu \|\nabla_j(\E^i f)\|^2.
\end{multline*}
Equivalently, we can write
\begin{multline*} 
(1-4kC_0 J_{ij}^2) \, \nu \|\nabla_j(\E^i f)\|^2  
\le (2+4kJ_{ij}^2 C_0)\nu\|\nabla_j f\|^2+  2kJ_{ij}^2 (c_p+C_0+2c_pC_0) \nu \|\nabla_i f\|^2.
\end{multline*}
We now need to make sure that $1-4kC_0 J^2 > 0$, i.e., 
that $J < (4kC_0)^{-1/2}$ and that $2kJ_{ij}^2 (c+C_0+2cC_0)  /(1-4kC_0 J^2)<1$,
that is, $2kJ_{ij}^2 (c_p+C_0+2c_pC_0) +4kC_0 J^2<1$, or $J<(2kc_p+4kc_pC_0+6kC_0 )^{-1/2}$. 
But the latter inequality implies the former. So it is only the
latter that we need.
Therefore the inequality holds with
$D_1 := (2+4kC_0 J^2)/(1-4kC_0 J^2)$
and
$D_2 := 2kJ_{ij}^2 (c_p+C_0+2c_pC_0)$,
provided that $J<(2kc_p+4kc_pC_0+6kC_0 )^{-1/2}$.
\end{proof}

\begin{cor}\label{coers2}
Assume (\ref{H2.1})-(\ref{Hgeo}). For some $D_3 > 0$, if $i, j$ are neighbors in $\Z^d$, then
\[
\nu[(f-\E^i f)^2 \dd(x_j)^s] \le  D_3 \nu \|\nabla_j f\|^2
+ D_3 \nu\|\nabla_i f\|^2
\]
and
\[
\nu[(f-\E^i f)^2 \dd(x_i)^s] \le  D_3 \nu\|\nabla_i f\|^2.
\]\end{cor}
\begin{proof}
For the first  assertion,  replace $\nu\|\nabla_j(\E^i f)\|^2$ on the right-hand side of \eqref{E4.7} 
by its upper bound from the inequality  in the statement of Lemma \ref{sweep1},
and bound the last term 
from the spectral gap inequality from Lemma \ref{spectralgap1d}. Similarly, the second assertion of the corollary follows from (\ref{E4.6}) and Lemma \ref{spectralgap1d}, for a constant  $D_3:=2C_{0}(4+c_p)$.
 \end{proof}

Next, let, for $r =0,1, \ldots, d-1$, the set $\Gamma_r$
be defined by
\[
\Gamma_r := \{i \in \Z^d:\, i_1+\cdots+i_d \equiv r \mod d\}.
\]
Note that the sets $\Gamma_r$, $r =0,1, \ldots, d-1$,
form a partition of $\Z^d$ and
$\inf\{\max_{1\le k\le d} |i_k-j_k|:\, i \in \Gamma_r, j \in \Gamma_s\}=1$
if $r \neq s$.

From now on, we shall work with the case $d=2$, for simplicity of notation.
The general case is analogous.

\begin{lem}\label{sweepGam1}Assume (\ref{H2.1})-(\ref{Hgeo}).
There are constants $R_1>0$ and $0<R_2 < 1$ such that
 \[
\nu \|\nabla_{\Gamma_0} (\E^{\Gamma_1} f)\|^2 
 \le R_1 \nu \| \nabla_{\Gamma_0} f\|^2 + R_2 \nu \| \nabla_{\Gamma_1} f\|^2\]
and
\[
\nu \|\nabla_{\Gamma_1} (\E^{\Gamma_0} f)\|^2 
 \le R_1 \nu \| \nabla_{\Gamma_1} f\|^2 + R_2 \nu \| \nabla_{\Gamma_0} f\|^2.
\]
\end{lem}
\begin{proof}
Fix $i \in \Gamma_1$. Denote by $\partial\{i\}$ the set $\{i\pm e_1, i\pm e_2\}$
of the $2d=4$ neighbors of $i$. Since $\partial \{i\} \subset \Gamma_0$,
we can write $\E^{\Gamma_0} f = \E^{\Gamma_0\setminus \partial\{i\}} \E^{\partial\{i\}}f$.
Hence if $X^\alpha_i$ is one of the H\"ormander generators
 of $\Spin^{\{i\}}$, we have
$X^\alpha_i \E^{\Gamma_0} f = \E^{\Gamma_0\setminus \partial\{i\}} X^\alpha_i(\E^{\partial\{i\}}f)$.
By Jensen's inequality,
$(X^\alpha_i \E^{\Gamma_0} f)^2 
\le \E^{\Gamma_0 \setminus \partial\{i\}}[ (X^\alpha_i(\E^{\partial\{i\}}f))^2]$.
Summing over all $\alpha$, we get
$\|\nabla_i \E^{\Gamma_0} f\|^2 
\le \E^{\Gamma_0 \setminus \partial\{i\}} \|\nabla_i(\E^{\partial\{i\}}f)\|^2$.
Integrating over $\nu$ and using $\nu \E^{\Gamma_0 \setminus \partial\{i\}}=\nu$, we get
$\nu \|\nabla_i \E^{\Gamma_0} f\|^2
\le \nu \|\nabla_i(\E^{\partial\{i\}}f)\|^2$.
Summing this over $i \in \Gamma_1$ we have
\[
\nu\|\nabla_{\Gamma_1} (\E^{\Gamma_0} f)\|^2
\le \sum_{i\in \Gamma_1} \nu \|\nabla_i(\E^{\partial\{i\}}f)\|^2.
\]
We estimate the term inside the sum using Lemma \ref{sweep1}  as follows.
First let $\partial\{i\} = \{i+e_1, i+e_2, i-e_1,i-e_2\}
=\{j_1, j_2, j_3, j_4\}$.
Then $\nabla_i (\E^{\partial\{i\}}f) 
= \nabla_i \E^{\{j_1\}} \E^{\{j_2,j_3,j_4\}} f$.
So
\[
\nu\|\nabla_i (\E^{\partial\{i\}}f)\|^2
= \nu \|\nabla_i \E^{\{j_1\}}  \E^{\{j_2,j_3,j_4\}} f\|^2
\le D_1 \nu\|\nabla_i \E^{\{j_2,j_3,j_4\}} f\|^2
+ D_2 \nu\|\nabla_{j_1} \E^{\{j_2,j_3,j_4\}} f\|^2.
\]
For the second term we have 
$\nabla_{j_1} \E^{\{j_2,j_3,j_4\}} f
= \E^{\{j_2,j_3,j_4\}} \nabla_{j_1} f$ and so, by Jensen's inequality,
\[
\nu\|\nabla_{j_1} \E^{\{j_2,j_3,j_4\}} f\|^2
\le \nu \E^{\{j_2,j_3,j_4\}} \|\nabla_{j_1} f\|^2
= \nu  \|\nabla_{j_1} f\|^2.
\]
The first term is estimated using Lemma \ref{sweep1} once more:
\[
\nu\|\nabla_i \E^{\{j_2,j_3,j_4\}} f\|^2
= \nu  \|\nabla_i \E^{j_2} \E^{\{j_3,j_4\}} f\|^2
\le D_1 \nu\|\nabla_i \E^{\{j_3,j_4\}} f\|^2
+ D_2 \nu \|\nabla_{j_2} \E^{\{j_3,j_4\}} f\|^2.
\]
Continuing in this manner, we obtain (observe that $D_1 > 1$)
\begin{multline*}
\nu\|\nabla_i (\E^{\partial\{i\}}f)\|^2
\\
\le D_1^4 \nu\|\nabla_i f\|^2
+ D_1^3 D_2 \nu \|\nabla_{j_4} f\|^2
+ D_1^2 D_2 \nu \|\nabla_{j_3} f\|^2
+ D_1 D_2 \nu \|\nabla_{j_2} f\|^2
+ D_2 \nu \|\nabla_{j_1} f\|^2
\\
\le D_1^4 \nu\|\nabla_i f\|^2 + D_1^3 D_2 \sum_{j \in \partial\{i\}}
\|\nabla_j f\|^2.
\end{multline*}
Summing up over all $i \in \Gamma_1$,
\[
\nu \|\nabla_{\Gamma_1}(\E^{\Gamma_0} f)\|^2
\le D_1^4 \nu\|\nabla_{\Gamma_0} f\|^2 
+ 4 D_1^3 D_2 \nu\|\nabla_{\Gamma_1} f\|^2.
\]
We need to make sure that $4D_1^3 D_2 < 1$. Substituting the
actual expressions for these constants we can see that this
inequality is satisfied for all sufficiently small positive $J$.
In particular, the inequality is true
for all $J < (80 k(c+2cC_0+2C_0))^{-1/2}$.
We have thus proved the second inequality with $R_1 :=D_1^4$
and $R_2 := 4 D_1^3 D_2$, provided that
$J < (192( k^2+k)(c+2cC_0+2C_0+C_0^3))^{-1/2}$.
\end{proof}
Define now the symbol $\Q^n$ to be $\Q^0 f=f$ and $\Q^n := \E^{\Gamma_0} \Q^{n-1}$ when $n$ is odd and  $\Q^n := \E^{\Gamma_1} \Q^{n-1}$ when $n$ is even, 
with the understanding that $\Q^{n}$ when $n$ is even takes a functional $g$ on $\Spin^{\Z^d}$,
integrates with respect to $\P^{\Gamma_1, x_{\partial \Gamma_0}}(dx_{\Gamma_1})
= \P^{\Gamma_1, x_{\Gamma_0}}(dx_{\Gamma_1})$
 so that $\Q^{n} g$ is a functional not depending on $x_{\Gamma_1}$. Analogously,   $\Q^{n} g$ for $n$ odd is a functional not depending on $x_{\Gamma_0}$. We used the fact that $\partial \Gamma_0 = \Gamma_1$ and
$\partial \Gamma_1 = \Gamma_0$.

\begin{lem}\label{conv}
Under hypotheses (\ref{H2.1})-(\ref{Hgeo}),  we  have
that $\lim_{n \to \infty} \Q^n f = \nu f$, $\nu$-a.e.
\end{lem}

\begin{proof}
We will estimate the $\LL^2(\nu)$ norm of the differences of $\Q^n f$.
From the spectral gap inequality for $\E^{\Gamma_k},k=0,1$ (which follows
from the product property of the spectral gap and the spectral gap for the one node from  Lemma \ref{spectralgap1d}) we have
\[\E^{\Gamma_k}(\Q^{n} f-  \Q^{n+1} f)^2=\E^{\Gamma_k}(\Q^{n} f-\E^{\Gamma_k} \Q^{n} f)^2 \le c_p\,\E^{\Gamma_0}\|\nabla_{\Gamma_k}\Q^{n} f\|^2.\]
Integrating with respect to $\nu$ we have
\[\nu(\Q^{n} f-\E^{\Gamma_k} \Q^{n} f)^2 \le c_p\,\nu\|\nabla_{\Gamma_k}\Q^{n} f\|^2.\]
The last term is estimated from Lemma \ref{sweepGam1}, for $n \ge 2$,
\[
\nu[(\Q^n f-\Q^{n+1} f)^2] \le c_p (R_1 +R_{2})R_2^{n-1} \nu \|\nabla_{ }   f\|^2 \le R^n,
\]
for some $R$ (depending on $f$), with $0<R<1$.
Let $\epsilon>0$ be so small so that $R(1+\epsilon)<1$.
Then
\[
\nu\{x\in \Omega:\, |\Q^n f-\Q^{n+1}f| > (R(1+\epsilon))^{n/2}\}
\le \nu[(\Q^n f-\Q^{n+1} f)^2]/(R(1+\epsilon))^n \le (1+\epsilon)^{-n}.
\]
Hence
\[
\nu\{x\in \Omega:\, |\Q^n f - \Q^{n+1} f| \le (R(1+\epsilon))^{n/2}
\text{ for almost all } n\} = 1.
\]
By the triangle inequality,
\[
\nu\{x\in \Omega:\, |\Q^n f - \Q^m f| \le (R(1+\epsilon))^{n/2}/
(R(1+\epsilon))^{1/2} \text{ for all large $n$ and $m$}\}=1.
\]
Hence $\Q^n f$ converges $\nu$-a.e. say to,  $  \xi (f)$.      At first we will show that $\xi(f)$  is a constant  that  does not depend on variables neither on  $\Gamma_0$ nor on $\Gamma_1$. We first observe that  $\Q^n(f)$ is a function on $\Gamma_1$ or $\Gamma_0$ when $n$  is odd or even respectively. As a consequence the limits of the subsequences   $\lim_{n \text{\ odd}, n\rightarrow \infty}\Q^nf$  and $\lim_{n \text{\ even}, n\rightarrow \infty}\Q^nf$ do not depend on  variables on $\Gamma_0$ and $\Gamma_1$ respectively.    However, since  the two  subsequences $\{\Q^nf\}_{n\text{\ even}}$ and $\{\Q^nf\}_{n\text{\ odd}}$  converge to the same limit $\xi(f)$ $\nu-$a.e.  we conclude that \[\lim_{n \text{\ odd}, n\rightarrow \infty}\Q^nf=\xi(f)=\lim_{n \text{\ even}, n\rightarrow \infty}\Q^nf\]
 from which we derive that $\xi(f)$ is a constant. Furthermore, this implies that 
\[\nu \left(\xi\ (f) \right)=\xi(f).\]
To finish the proof, it remains to show that  $\xi(f)=\nu (f)$. One notices that since the sequence $\{\Q^n f\}_{n \in \mathbb{N}}$ 
 converges $\nu-$a.e, the same holds for  the sequence $\{\Q^n f-\nu \Q^n f\}_{n \in \mathbb{N}}$. 

At first assume  positive bounded functions  $f$. In this case we have 
 \[\lim_{n\rightarrow \infty}(\Q^n f-\nu \Q^n f)=\xi (f)-\nu\left(\xi(f) \right) =\xi (f)-\xi (f)=0\]
 by the dominated     convergence theorem and the fact that $\xi(f)$ is constant. On the other hand, we also have 
 \[\lim_{n\rightarrow \infty}(\Q^n f-\nu \Q^nf)=\lim_{n\rightarrow \infty}(\Q^n f-\nu f)=\xi\ (f)-\nu (f)
\] 
by the definition of the Gibbs measure $\nu$. From the last two we obtain $\xi (f)=\nu (f)$
for bounded positive functions $f$. We will extend this to   unbounded positive functions $f$. For this we consider $f_{k}(x):=\max\{f(x),k\}$ for any $k\in \mathbb{N}$. Then  \[\xi(f_{k})=\lim_{n\rightarrow \infty}\Q^nf_{k}= \nu f_{k}  \] $\nu$ a.e, since $f_{k}(x)$ is bounded by $k$. 
 But since $f^k$ is increasing on $k$, by the monotone convergence theorem we   get 
\begin{align*}\xi(f)=\lim_{k\rightarrow \infty}\xi(f_{k})=\lim_{k\rightarrow \infty}\nu(f_{k})=\nu(\lim_{k\rightarrow \infty}f_{k})=\nu(f)  \  \   \nu \  \  a.e. \end{align*}
The assertions   can be  extended to non-positive functions $f$ just by writing $f=\max\{f,0\}-\min\{f,0\}$.
\end{proof}

 \section{log-Sobolev inequality for one site measure.} \label{secLS}

In this section we show a weak version of the  log-Sobolev type inequality for the one site measure $\mathbb{E}^{i,\omega}$.

\begin{prop}\label{prop5.1}  Assume (\ref{H2.1})-(\ref{Hgeo}) and  that the  measure $\mu $ satisfies the log-Sobolev inequality with a constant $c$. Then, for $J$ sufficiently small uniformly on $\omega$,   the one site measure  $\mathbb{E}^{i,\omega}$ satisfies the following weak version of  a log-Sobolev  inequality \[\nu\mathbb{E}^{i,\omega}\left(  f^2\log\frac{f^2}{\mathbb{E}^{i,\omega}f^2}\right)\leq c_{1}\nu\|\nabla_i f\|^2 +c_{2}\sum_{ j\sim i}  \nu \|\nabla_j f\|^2  \]
for some    positive constants $c_{1}$ and $c_2<1$.
\end{prop}          

\begin{proof} 
  We begin with the main assumption on  the measure $\mu(dx_i)=\frac{e^{\phi(x_i)}dx_i}{\int \phi(x_i)dx_i}$, that it  satisfies a  log-Sobolev inequality with a constant $c$   
\[\mu  (f^2\log\frac{f^2}{\mu f^2})\leq c\mu \| \nabla_i f \|^2\]
We will interpolate the phase $\phi$ by the interactions $W^i:=\sum_{j \sim i}J_{ij}V(x_i,\omega_j)$ in order to form the Hamiltonian     of the   one site measure $\mathbb{E}^{ i,\omega}$. To achieve  this,   replace $f$ by      $e^\frac{-W ^i}{2} f$, 
 \begin{align}\int e^{-H^i} &f^2\log\frac{e^{-W^i} f^2}{\int (e^{-H^i} f^2)dx_{i} /  \int e^{-\phi(x_i)}
 dx_{i}})dx_{i} 
  \label{E5.1}\leq c  \int e^{- \phi(x_i)} \| \nabla_{i} (e^\frac{-W^i}{2} f)
\|^2 dx_{i}.\end{align}
 \noindent We denote  by  $D_l$ and $D_r$  the  left and right hand side of~\eqref{E5.1} respectively.
Use the Leibnitz rule
for the gradient on $D_r$, to bound $\| \nabla_{i} (e^\frac{-W^i}{2} f) \|^2 \leq 2  e^{-W^i}\| \nabla_{i}  f \|^2+ \frac{1}{2}  e^{-W^i}f^2\| \nabla_{i}  W^i \|^2 $, so that  
 \begin{align}      D_r\leq   \label{E5.2} \left( \int e^{-H^i}dx_{i} \right) \left( 2c\mathbb{E}^{ i,\omega}\|\nabla_{i} f
\|^2+\frac{ c}{2}\mathbb{E}^{ i,\omega}(f^{2}  \|  \nabla_{i}{W^i}
\|^2)\right).\end{align}
On the left hand side of (\ref{E5.1}) we form the Hamiltonian $H^i=\phi(x_i)+W^i$ to  obtain the entropy for the measure $\mathbb{E}^{ i,\omega}$
\begin{align*}\nonumber D_l = &\int e^{-H^i} f^2log\frac{f^2}{\int e^{-H^i}f^2dx_{i} / \int e^{-H^i}
 dx_{i}} dx_{i}\nonumber +\int e^{-H^i} f^2\log\frac{\left(\int e^{-\phi(x_i)}
 dx_{i}\right)  e^{-W^i}}{ \int e^{-H^i}
 dx_{i}} dx_{i}\nonumber
 \\= &\nonumber \left( \int e^{-H^i}
 dx_{i} \right) \left(\mathbb{E}^{ i,\omega}(f^2log\frac{f^2}{ \mathbb{E}^{ i,\omega}f^2}) - \mathbb{E}^{i,\omega}( f^2W^i)\right)+\int e^{-H^i} f^2\log\frac{\int e^{-\phi(x_i)}
 dx_{i}   }{ \int e^{-H^i}
 dx_{i}} dx_{i}. \end{align*}
Since  $W^i$ is no negative,  the  last  gives 
 \begin{align}\label{E5.3} D_l \geq \left( \int e^{-H^i}
 dx_{i} \right)\left(\mathbb{E}^{ i,\omega}(f^2log\frac{f^2}{ \mathbb{E}^{ i,\omega}f^2}) - \mathbb{E}^{ i,\omega}( f^2W^i)\right). \end{align}
Combining~\eqref{E5.1} together    with~\eqref{E5.2}  and~\eqref{E5.3} we obtain
\begin{align}\label{E5.4}\mathbb{E}^{ i,\omega}(f^2log\frac{f^2}{\mathbb{E}^{ i,\omega}f^2}) \leq &2c \mathbb{E}^{ i,\omega}\| \nabla_{i} f
\|^2+\mathbb{E}^{ i,\omega}\left(f^2(\frac{ c}{2} \|  \nabla_{i}W^i
\|^2+W^i)\right) . \end{align}
 We now consider the following bound for the entropy, shown in   \cite{B-Z} and \cite{R}   
 \[\mathbb{E}^{ i,\omega}(f^{2}log\frac{f^{2}}{\mathbb{E}^{ i,\omega}f^{2}}) \leq   A\mathbb{E}^{ i,\omega}( f-\mathbb{E}^{ i,\omega}f )^{2} +
 \mathbb{E}^{ i,\omega}\left(( f-\mathbb{E}^{ i,\omega}f )^{2}log\frac{( f-\mathbb{E}^{ i,\omega}f )^{2}}{\mathbb{E}^{ i,\omega}( f-\mathbb{E}^{ i,\omega}f )^{2}}\right)\]
  for some positive constant $A$.  
 Use (\ref{E5.4}) to bound the entropy appearing on the second term on the right hand side,  
 \begin{align*} \nonumber\ \mathbb{E}^{ i,\omega}(f^{2}log\frac{f^{2}}{\mathbb{E}^{ i,\omega}f^{2}}) \leq &  A\mathbb{E}^{ i,\omega}( f-\mathbb{E}^{ i,\omega}f )^{2} +
 2c \mathbb{E}^{ i,\omega}\| \nabla_{i} f
\|^2+\\ & +\mathbb{E}^{ i,\omega}\left(( f-\mathbb{E}^{ i,\omega}f )^{2}(\frac{ c}{2} \|  \nabla_{i}W^i
\|^2+W^i)\right).\end{align*}
If we take expectations with respect to  the Gibbs measure we have 
\begin{align*} \nonumber\ \nu(f^{2}log\frac{f^{2}}{\mathbb{E}^{ i,\omega}f^{2}}) \leq &  A\nu( f-\mathbb{E}^{ i,\omega}f )^{2} +
 2c \nu\| \nabla_{i} f
\|^2+\\ & +J\sum_{j\sim i}\nu\left(( f-\mathbb{E}^{ i,\omega}f )^{2}\{2c \|  \nabla_{i}V(x_i,\omega_j)
\|^2+V(x_i,\omega_j)\}\right)
 \end{align*}where above we use that $J_{i,j}^2\leq J_{i,j}\leq J$.  And so, from the bounds (\ref{H2.8}) and  (\ref{H2.9}) 
 \begin{align*} \nonumber\nu(f^2&log \frac{f^2}{\mathbb{E}^{ i,\omega}f^2}) \leq  (A+4Jk(1+2c))\nu( f -\mathbb{E}^{ i,\omega}f)^2  +2c \nu\| \nabla_{i} f
\|^2+\\  &4(2c+1)kJ\nu(( f-\mathbb{E}^{i,\omega}f)^{2}  \dd^s(x_i) ) +(2c+1)kJ\sum_{ j\sim i} \nu(( f-\mathbb{E}^{i,\omega}f)^{2}  \dd^s(\omega_j) ).\end{align*}
We bound the variance in the first  term by the spectral gap of Lemma \ref{spectralgap1d} 
and the third and the  fourth term   by Corollary \ref{coers2}
 \begin{align*} \nonumber\nu(f^2\log \frac{f^2}{\mathbb{E}^{ i,\omega}f^2}) \leq & \left(  (A+4Jk(1+2c))c_{p}    +2c+(16c+8)kJ  D_3  \right)\nu\| \nabla_{i} f
\|^2+\\  &+(2c+1)kJ D_3\sum_{ j\sim i} \nu \|\nabla_j f\|^2
 \end{align*}
which finishes the proof of  the proposition for $c_1=   (A+4Jk(1+2c))c_{p}    +2c+(16c+8)kJ  D_3$ and $c_2=(2c+1)kJ 2C_{0}(4+c_p)<1$ for $J<((2c_0+1)k2C_{0}(4+c_p))^{-1}$.\qed\end{proof}

\section{Further sweeping-out inequalities}\label{section6}
In this section we prove the second set of sweeping-out inequalities.
\begin{lem}\label{sweep2}
Assume (\ref{H2.1})-(\ref{Hgeo}) and the log-Sobolev inequality for $\mu$. If $i \sim j$ then for some $G_1>0$ and $0<G_2<1$,
\[
\nu\| \nabla_i \sqrt{\E^j f^2}\|^2 
\le G_1 \nu \|\nabla_i f\|^2 + G_2 \nu\|\nabla_j f\|^2
\]
\end{lem}
\begin{proof}
Fix neighboring sites $i, j$.
Start with the left-hand side, 
$$\nu \| \nabla_i \sqrt{\E^j f^2}\|^2 = \sum_{\alpha=1}^n
(X_i^\alpha \sqrt{\E^j f^2})^2,$$ where 
\begin{equation}
\label{E6.1}
(X_i^\alpha \sqrt{\E^j f^2})^2 
= \frac{(X_i^\alpha \E^j f^2)^2}{4\E^j f^2},
\end{equation} 
estimate the numerator as in \eqref{E4.1}: 
\[
(X_i^\alpha \E^j f^2)^2
\le 2 (\E^j(X_i f^2))^2 + 2(\int(X_i^\alpha \rho_j) f^2 dx_j)^2.
\]
Use Leibnitz' rule, Cauchy-Schwarz and Jensen 
for the first summand and estimate
the second using \eqref{E4.2} and \eqref{E4.4}:
\[
(X_i^\alpha \E^j f^2)^2  \,
\le 4 (\E^j f^2) \E^j(X_i^\alpha f)^2
+ 2 J_{ji}^2 \cov_{\E^j}[f^2,\, X_i^\alpha V(x_j, x_i)] ^2,
\]
where $\cov_\mu(f, g) = \mu(fg)-\mu(f)\mu(g) = \mu(f(g-\mu g))$, 
for a probability measure $\mu$.
Substituting into \eqref{E6.1} and summing over $\alpha$, we get
\[
\|\nabla_i \sqrt{\E^j f^2}\|^2  \le \E^j\|\nabla_i f\|^2 + \frac{J^2}{2} 
\sum_\alpha \frac{\cov_{\E^j}[f^2,\, X_i^\alpha V(x_j, x_i)] ^2}{\E^j f^2}.
\]
Instead of using Jensen,  as we did in \eqref{E4.3}, we use
the following inequality (see Lemma 4.1 in \cite{Pa1}):
\begin{lem} For a probability measure $\mu$
\[
(\cov_\mu(f^2, g))^2 \le 8\, (\mu f^2) \mu [(f-\mu f)^2 (g^2 + \mu g^2)].
\] 
\end{lem}
We get
\[
\|\nabla_i \sqrt{\E^j f^2}\|^2  \le \E^j\|\nabla_i f\|^2 + 4J^2  
\E^j\big \{(f-\E^j f)^2 (\|\nabla_i V(x_j,x_i)\|^2 + \E^j \|\nabla_i V(x_j,x_i)\|^2) \big\}.
\]
If we now use  condition  (\ref{H2.8})  to bound the interactions, and then take   expectations with  respect to $\nu$  we obtain 
\begin{align}\nonumber
\nu \|\nabla_i \sqrt{\E^j f^2}\|^2  \le &\nu \|\nabla_i f\|^2 +8k J^2  
\nu[\E^j (f-\E^j f)^2 ]+ 4 k J^2  
\nu[\E^j[(f-\E^j f)^2 \E^j \dd(x_j)^s]]+ \\  & +8k J^2  
\nu [(f-\E^j f)^2\dd(x_i)^s] +  4 k J^2  
\nu[ (f-\E^j f)^2   \dd(x_j)^s ].\label{E6.2}
\end{align}
At first notice that from Lemma \ref{coers} we can bound $\E^j[\dd(x_j)^r] \le C_0$. So the sum of the second and third term can be bounded from the variance with respect to the one site measure $\E^i$. Then the variance can be bounded by the spectral gap inequality obtained in Lemma \ref{spectralgap1d}. 
\begin{align*}8\lambda J^2  
\nu[\E^j (f-\E^j f)^2 ]+ 4 k J^2  
\nu[\E^j[(f-\E^j f)^2 \E^j \dd(x_j)^s]] \leq & 4kJ^2  (2+C_0)\nu [\E^j (f-\E^j f)^2]\\ \leq & 4k J^2  (2+C_0) c_p\nu \|\nabla_j f\|^2.\end{align*}
For the  remaining two last  terms in the right hand side of (\ref{E6.2}), we can use  the two bounds presented in  Corollary \ref{coers2}. 
If we put all these bounds together we get
\begin{align*}\nonumber
\nu \|\nabla_i \sqrt{\E^j f^2}\|^2  \le &(1+8k J^2  D_3)\nu \|\nabla_i f\|^2 +4k J^2(3
 D_3+  (2+C_0) c_p)\nu \|\nabla_j f\|^2.
\end{align*}
This proves the lemma with constants
$G_1 = 1+8k J^2  D_3$
and $G_2 = 4k J^2(3
 D_3+  (2+C_0) c_p)<1$, provided that 
 $J < (4k(3
 D_3+  (2+C_0) c_p))^{-\frac{1}{2}} $.\qed
\end{proof}

\begin{lem}\label{sweepGam2}
Assume (\ref{H2.1})-(\ref{Hgeo}) and the log-Sobolev  inequality for $\E^{i,\omega}$. There are constants $C_1>0$ and $0<C_2 < 1$ such that
 \begin{align*}
\nu\|\nabla_{\Gamma_0} \sqrt{\E^{\Gamma_1} h^2}\|^2
&\le C_1\, \nu\|\nabla_{\Gamma_0} h\|^2 + C_2\, \nu\|\nabla_{\Gamma_1} h\|^2
\end{align*}
and
\begin{align*}
\nu\|\nabla_{\Gamma_1} \sqrt{\E^{\Gamma_0} h^2}\|^2
&\le C_1\, \nu\|\nabla_{\Gamma_1} h\|^2 + C_2\, \nu\|\nabla_{\Gamma_0} h\|^2.
\end{align*}
\end{lem}
\begin{proof}
We will make frequent use of the following inequality. Let $A, B$
be subsets of $\Z^2$ at lattice distance at least $2$ 
and $i \in \Z^2$ such that 
$\partial\{i\} \cap A = \varnothing$.
Then
\[
\nu\|\nabla_i \sqrt{\E^{A\cup B} f^2}\|^2
\le \nu\|\nabla_i \sqrt{\E^{B} f^2}\|^2.
\]
To see this, let $\nabla_i = (X^\alpha_i, \alpha=1,\ldots,n)$
and write
\[
X^\alpha_i \sqrt{\E^{A\cup B} f} 
= \frac{X^\alpha_i \E^{A\cup B} f}{2\sqrt{\E^{A\cup B}f}}
= \frac{\E^A X^\alpha_i \E^B f}{2\sqrt{\E^{A\cup B}f}}
= \frac{2 \E^A [ \sqrt{\E^B f} X^\alpha_i \sqrt{\E^B f}]}{2\sqrt{\E^{A\cup B}f}},
\]
where the first and last inequalities are due to Leibnitz' rule,
while the middle one follows from the assumptions on $A$, $B$ and $i$.
By Cauchy-Schwarz, $(\E^A [ \sqrt{\E^B f} X^\alpha_i \sqrt{\E^B f}])^2
\le (\E^A \E^B f) \, \E^A(X^\alpha_i \sqrt{\E^B f})^2$. 
Squaring the last display and  replacing by this inequality we obtain
$(X^\alpha_i \sqrt{\E^{A\cup B} f})^2 \le \E^A(X^\alpha_i \sqrt{\E^B f})^2$.
Summing over $\alpha$ and integrating over $\nu$ proves the claim.

To save some space below, for $F : \Spin^{\Z^d} \to \R^n$
we shall write $\III{F}^2$ instead of $\int \|F(x)\|^2 d\nu(x)$.
We shall also write $\widehat E f$ instead of $\sqrt{\E f}$. Thus the inequality
we showed is written as 
\[
\III{\nabla_i \widehat\E^{A\cup B} f^2} ^2
\le \III{\nabla_i \widehat\E^{B} f^2}^2
\tag{QS}.
\]
Using this we upper bound $\nu\|\nabla_{\Gamma_1} \sqrt{\E^{\Gamma_0}f^2}\|^2$:
\begin{equation}
\label{E6.3}
\III{\nabla_{\Gamma_1} \widehat\E^{\Gamma_0} f^2}^2
= \sum_{i\in\Gamma_1} \III{\nabla_i \widehat\E^{\Gamma_0}f^2}^2
\le \sum_{i\in\Gamma_1} 
\underbrace{\III{\nabla_i \widehat\E^{\partial\{i\}}f^2}^2}_{:=T_1(i)}.
\end{equation}
Fix $i \in \Gamma_1$ and denote its neighbors by $i_1, i_2,i_3,i_4$. Let also $I := \{i_2,i_3,i_4\} = \partial\{i\}\setminus\{i_1\}$.
Using Lemma \ref{sweep2} we write
\begin{equation*}
T_1(i)
=\III{\nabla_i  \widehat\E^{i_1}\E^I f^2}^2
\le G_1 \underbrace{\III{\nabla_i \widehat\E^I f^2}^2}_{:=T_2(i)}
+ G_2 \III{\nabla_{i_1} \widehat\E^I f^2}^2.
\end{equation*}
Using (QS) three times in the second term, we obtain 
\[
\III{\nabla_{i_1} \widehat\E^I f^2}^2 \le \III{\nabla_{i_1} f}^2.
\]
And so,
\begin{align}  T_1(i)\leq G_1 T_2(i)+   G_2  \sum_{\ell \sim i}\III{\nabla_{\ell} f}^2  \label{E6.4}
.\end{align}  
   Now we sum  over $i \in \Gamma_1$. Note that 
$\sum_{i \in \Gamma_1} \sum_{\ell \sim i}\III{\nabla_{\ell} f}^2
=4 \sum_{j\in \Gamma_0} \III{\nabla_j f}^2$.
\[
 \sum_{i \in \Gamma_1} T_1(i) 
\le G_1 \sum_{i \in \Gamma_1} T_2(i)
+   4G_2\III{\nabla_{ \Gamma_0} f}^2.
\]
We proceed in the same manner to estimate $T_2(i)$.
Let $J=\{i_3,i_4\}$,
\begin{equation}
\label{E6.5}
 T_2(i) 
:= \III{\nabla_i \widehat\E^{i_2} \E^J f^2}^2
\le 
G_1 \underbrace{\III{\nabla_i \widehat\E^J f^2}^2}_{:=T_3(i)}
+ G_2 \III{\nabla_{i_2} \widehat\E^J f^2}^2.
\end{equation}
Use (QS) for the second term,
\[
\III{\nabla_{i_2} \widehat\E^J f^2}^2 
\le \III{\nabla_{i_2}f}^2.
\]
Substituting  into \eqref{E6.5} 
\begin{equation}
 T_2(i) 
\le 
G_1 T_3(i)+  G_2  \sum_{\ell \sim i}\III{\nabla_{\ell} f}^2\label{E6.6}
\end{equation}
and summing up over $i \in \Gamma_1$,
\[
 \sum_{i\in \Gamma_1} T_2(i)
\le G_1 \sum_{i\in \Gamma_1} T_3(i) 
+  4  G_2  \III{\nabla_{ \Gamma_0} f}^2 .
\]
The next term is similar:
\[
T_3(i) 
= \III{\nabla_i \widehat\E^{i_3} \E^{i_4} f^2}^2
\le G_1 \III{\nabla_i \widehat\E^{i_4}f^2}^2
+ G_2 \III{\nabla_{i_3} \widehat\E^{i_4}f^2}^2,
\]
with the  terms  estimated as
\begin{align*}
\III{\nabla_i \widehat\E^{i_4}f^2}^2 &\le
G_1 \III{\nabla_{i} f}^2
+ G_2 \III{\nabla_{i_4} f}^2
 \\
\III{\nabla_{i_3} \widehat\E^{i_4}f^2}^2 &\le \III{\nabla_{i_3} f}^2,
\end{align*}
so that
\begin{align}
T_3(i) 
\le  G^{2}_1 \III{\nabla_{i} f}^2
+ (1+G_{1} )G_2  \sum_{\ell \sim i}\III{\nabla_{\ell} f}^2 \label{E6.7}
\end{align}
and summing over $i\in\Gamma_1$ \[
 \sum_{i\in \Gamma_1} T_3(i)
\le G^{2}_1 \III{\nabla_{ \Gamma_1} f}^2
+ 4(1+G_{1} )G_2  \III{\nabla_{ \Gamma_0} f}^2 .
\]

Substituting the terms involving the sums to one another and then
back to \eqref{E6.3} yields the second inequality in
the statement with 
$C_1 =G_1^4 $ and  
$C_2 = 4G_2 (1+4G_1+G^{2}_{1}++G^{3}_{1}).$ 
Since $G_2 = 4k J^2(3 D_3+  (2+C_0) c_p)<1$ we can choose $J$ sufficiently small such that $G_2$ is small enough so that $C_2<1$.
\end{proof}

In the next proposition we prove a weak log-Sobolev inequality for the product measures $\E^{\Gamma_i}, i=0,1$, where the entropy of the measure $\E^{\Gamma_i}$ for $ i=0$ or $1$ is not bounded by the  relevant gradient, but from the gradients of both $\Gamma_0$ and  $\Gamma_1$. 
 
  \begin{prop}\label{LSGamma0}Assume (\ref{H2.1})-(\ref{Hgeo}) and the log-Sobolev inequality for $\mu$. Then the following log-Sobolev type inequality for the measure  $\mathbb{E}^{i,\omega}$ holds  
 \begin{align*}  \nu\mathbb{E}^{ \Gamma_{k}}(f^2\log\frac{f^2}{\mathbb{E}^{ \Gamma_{k}} f^2}) \leq\tilde C \nu\left\vert \nabla_{\Gamma_0}
f
\right\vert^2+\tilde C\nu\left\vert \nabla_{\Gamma_1} f
\right\vert^2\end{align*}for $k=0,1$, and  some positive constant $\tilde C$.\end{prop}    

\begin{proof}     Consider a node  $i\in \mathbb{Z}^2$ with  four neighbours denoted as  $\{ \sim i \}=i_1,i_2,i_3,i_4$ . We start by considering the following two quantities:  
\begin{align*}\Phi(i):=&\nu\|\nabla_{i}( \mathbb{E}^{\{i_{1} ,i_{2},i_3,i_4\}}f^{2})^{\frac{1}{2}}\|^{2} +\nu \|\nabla_{i}( \mathbb{E}^{\{ i_{2},i_3,i_4\}}f^{2})^{\frac{1}{2}}\|^{2}+\nu\|\nabla_{i}( \mathbb{E}^{\{ i_3,i_4\}}f^{2})^{\frac{1}{2}}\|^{2}+\\ &+\nu\|\nabla_{i}( \mathbb{E}^{\{ i_4\}}f^{2})^{\frac{1}{2}}\|^{2}\end{align*}
and
\begin{align*}\Theta(i):=    \nu\|\nabla_{i}  f\|^2+ \sum_{s \sim i}\   \nu\|\nabla_{s}  f\|^2 .\end{align*} From the estimates (\ref{E6.4}), (\ref{E6.6}) and (\ref{E6.7})  about the components of  the sum of $\Phi(i)$ in the proof of   Lemma  \ref{sweepGam2} together with   Lemma \ref{sweep2} we   surmise that there exists a constant  $R_3>0$   such that   
\begin{align}\label{E6.8} \Phi(i) \leq  R_{3}\Theta(i).
\end{align}  Starting from the neighbourhood of  $(0,0)$ we form a spiral  enumeration of all  nodes in $\Gamma_1$ as described below (see also depiction in figure \ref{fig2}).
\begin{figure}[h]
           \begin{center}
\epsfig{file=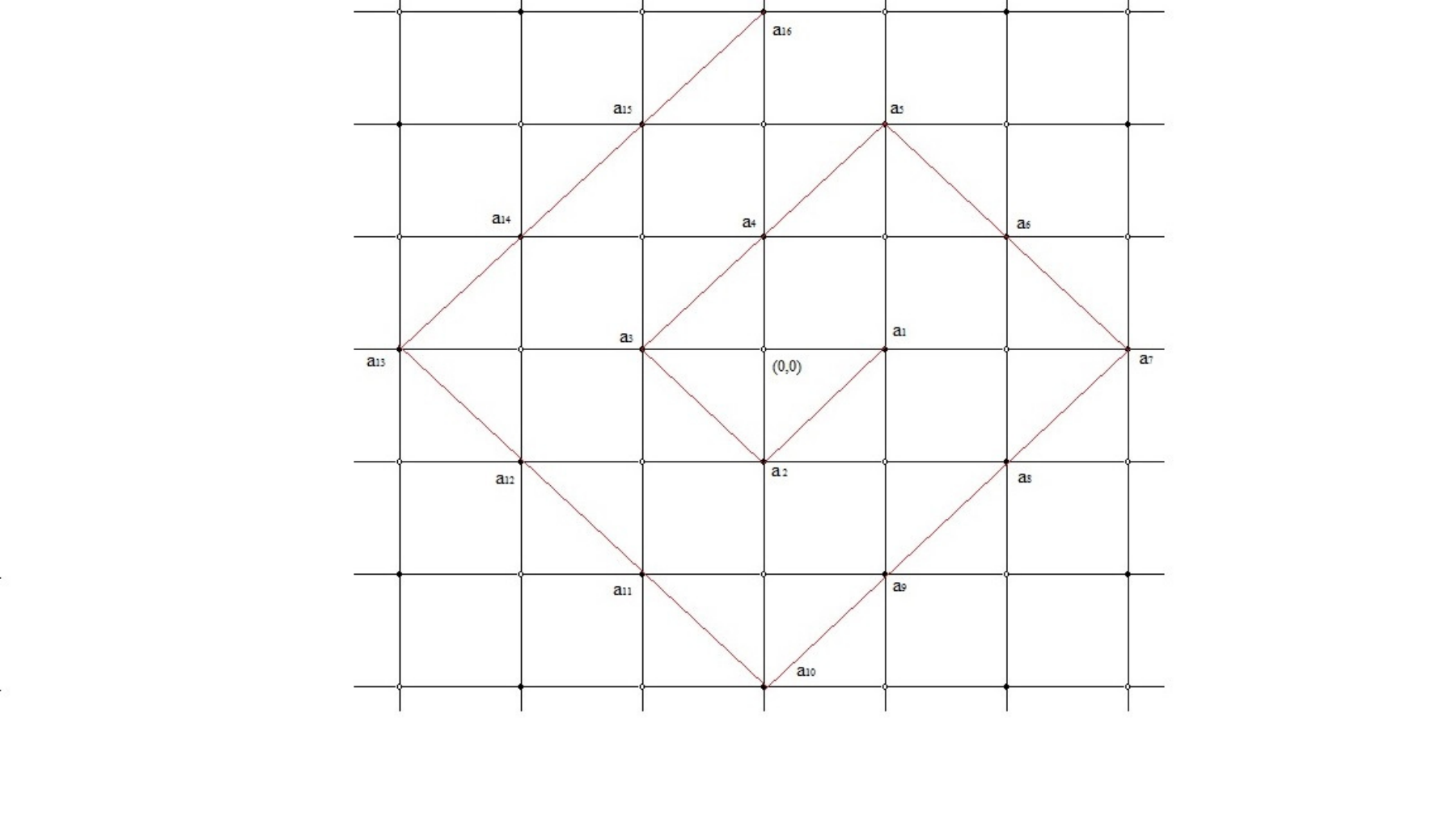, height=4.5cm}
\caption{$\circ = \Gamma_0$, $\bullet = \Gamma_1$}
\label{fig2}
           \end{center}
\end{figure}
 We start by denoting   $a_1,a_2,a_3,a_4$ the neighbours of $(0,0)$. Obviously, since $(0,0) \in \Gamma_0$, the nodes $a_i\in \Gamma_1$ for $i=1,...,4$. After choosing  $a_1$ from any of the four neighbours,  the rest are named clockwise. Then, we choose  $a_5$ to be any of the nodes in $\Gamma_1$ of distance two from $a_4$ and distance three from $(0,0)$.  We continue in the same manner   clockwise the enumeration of the rest of the nodes in $\Gamma_1$ that have distance  three from $(0,0)$, then distance four, and so on. In this way we   construct a spiral    comprising of the  nodes in $\Gamma_1$ always  moving clockwise while we move away  from $(0,0)$.    We can then write $\mathbb{E}^{ \Gamma_{1}}=\sqcap_{i=1}^{+\infty}\ \mathbb{E}^{ a_{i}}$.
Since we have obtain in Proposition \ref{prop5.1} a log-Sobolev inequality for the one node measure, for constants $c_1$ and $c_2$ uniformly on the node $i$ (recall that $c_2=(2c+1)kJ 2C_{0}(4+c_p)<1$ for $J<((2c_0+1)k2C_{0}(4+c_p))^{-1}$), we will express the entropy of the product measure $\E^{\Gamma_1}$ in terms of the individual entropies as seen  below
\begin{align}\nu\mathbb{E}^{ \Gamma_1}(f^2\log\frac{f^2}{\mathbb{E}^{ \Gamma_1}f^2}) = \label{E6.9}\sum_{k=1}^{+\infty}\nu\mathbb{E}^{a_k}(\mathbb{E}^{ a_{k-1}}...\mathbb{E}^{a_1}f^2\log\frac{\mathbb{E}^{a_{k-1}}...\mathbb{E}^{a_1}f^2}{\mathbb{E}^{a_k}...\mathbb{E}^{a_1}f^2})\end{align}
so that we can  upper bound the one site entropies from the log-Sobolev inequalities,   
 \begin{align}\label{H6.10} \nu\mathbb{E}^{  a_{k}}(\mathbb{E}^{ a_{k-1}}...\mathbb{E}^{ a_{1}}f^2\log\frac{\mathbb{E}^{a_{k-1}}...\mathbb{E}^{a_{1}}f^2}{\mathbb{E}^{a_{k}}...\mathbb{E}^{a_{1}}f^2}) \leq c_{1}\nu\|\nabla_{a_{k}}\ f\|^2 +c_{2}\sum_{ j\sim a_{k}}  \nu \|\nabla_j (\mathbb{E}^{ a_{k-1}}...\mathbb{E}^{ a_{1}}f^2)^\frac{1}{2}\|^2 \end{align}
 where above in the computation of the first term  we used that $a_i$'s have distance bigger than one from each other, and so $\nu\|\nabla_{a_k} (\mathbb{E}^{ a_{k-1}}...\mathbb{E}^{ a_{1}}f^2)^\frac{1}{2}\|^2 \leq \nu\|\nabla_{a_k} f\|^2$.  
 For the second summand in (\ref{H6.10}) notice that the neighbours of  $a_k$ can be distinguished into two categories. Those that have distance bigger than one from $a_{k-1},a_{k-2},...,a_1$ and those that neighbour with at least one of    $a_{k-1},a_{k-2},...,a_1$.  For $j\sim a_k$ that belong to the first category, since they do not neighbour any of the nodes  $a_{k-1},a_{k-2},...,a_1$ we clearly get
 \begin{align}\label{H6.11}\nu\|\nabla_{j} (\mathbb{E}^{ a_{k-1}}...\mathbb{E}^{ a_{1}}f^2)^\frac{1}{2}\|^2 \leq \nu\|\nabla_{j} f\|^2. \end{align}
 For those   neighbours of $a_k$, that neighbour  with at least one of the $a_{k-1},a_{k-2},...,a_1$ we can write 
\begin{align*}\nu\|\nabla_{j} (\mathbb{E}^{ a_{k-1}}...\mathbb{E}^{ a_{1}}f^2)^\frac{1}{2}\|^2 \leq\Phi(j).
\end{align*}
If we bound this by (\ref{E6.8}) 
\begin{align}\label{H6.12}\nu\|\nabla_{j} (\mathbb{E}^{ a_{k-1}}...\mathbb{E}^{ a_{1}}f^2)^\frac{1}{2}\|^2 \leq R_{3}\Theta(j).
\end{align}   
Gathering together    (\ref{H6.12}), (\ref{H6.11}) and  (\ref{H6.10}) we have
   \begin{align*} \nonumber\nu\mathbb{E}^{  a_{k}}(\mathbb{E}^{ a_{k-1}}...\mathbb{E}^{ a_{1}}f^2\log\frac{\mathbb{E}^{a_{k-1}}...\mathbb{E}^{a_{1}}f^2 }{\mathbb{E}^{a_{k}}...\mathbb{E}^{a_{1}}f^2}) \leq  & c_{1}\nu\| \nabla_{a_{k}} f\|^2+c_{2}\sum_{ j\sim a_{k}}  \nu\|\nabla_{j} f\|^2+\\ &  +c_{2}R_{3}\sum_{ j\sim a_{k}}  \Theta(j)  .\end{align*}
 Then, if we combine this bound  together with (\ref{E6.9}) we obtain 
  \begin{align*}\nu\mathbb{E}^{ \Gamma_1}(f^2\log\frac{f^2}{\mathbb{E}^{ \Gamma_1}f^2}) \leq  & c_{1}\sum_{k=1}^{+\infty}\nu\| \nabla_{a_{k}} f\|^2+c_{2}\sum_{k=1}^{+\infty}\sum_{ j\sim a_{k}}  \nu\|\nabla_{j} f\|^2+\\ &  +c_{2}R_{3}\sum_{k=1}^{+\infty}\sum_{ j\sim a_{k}}    \sum_{n=0}^{2}  \sum_{r:dist(s,j)=n}\   \nu\|\nabla_{s}  f\|^2.\end{align*}
If we notice that for every node there are four nodes at distance one and eight at distance two, after rearranging the sums above we finally obtain  
 \begin{align*}\nu\mathbb{E}^{ \Gamma_1}(f^2\log\frac{f^2}{\mathbb{E}^{ \Gamma_1}f^2}) \leq( c_{1}+13R_3c_{2})\nu\| \nabla_{\Gamma_1} f\|^2+( 4c_{2}+13R_3c_{2})\nu\| \nabla_{\Gamma_0} f\|^2.\end{align*}
  \end{proof}

\section{The log-Sobolev inequality for the Gibbs measure}\label{PROOF}

In this section we prove the main result stated in Theorem \ref{thmGENERAL}.
 We recall that  $\Q^n$ is defined as $\Q^0 f=f$ and $\Q^n := \E^{\Gamma_0} \Q^{n-1}$ when $n$ is odd and  $\Q^n := \E^{\Gamma_1} \Q^{n-1}$ when $n$ is even.\begin{proof}
If $\Lambda$ is a subset of $\Z^d$, we write
$\ent_{\E^\Lambda}$ for the entropy of the probability measure
$\P^{\Lambda,\omega}$ on $\Spin^\Lambda$, that is, 
$\ent_{\E^\Lambda}(g) =  \E^\Lambda \big[ g \log \frac{g}{\E^\Lambda g} \big]$.
From this, with $\lambda(x) := x \log x$, we have
\begin{equation}
\label{H7.1}
\E^\Lambda[\lambda(g)] = \ent_{\E^\Lambda}(g) + \lambda(\E^\Lambda g),
\end{equation}
where we used the fact that $\E^\Lambda g$  does
not depend on $x_{\Lambda}$. 

We claim that, for all $n \ge 1$,
\begin{align}\nonumber
\label{H7.2}
\Q^n[\lambda(g)] 
= &\sum_{m=0,\ m\ \text{odd} }^{n-1} \Q^{n-m-1} \E^{\Gamma_1}[\ent_{\E^{\Gamma_0}}(\Q^m g)]
+ \sum_{m=0,\  m\ \text{even}}^{n-1} \Q^{n-m-1} [\ent_{\E^{\Gamma_1}}( \Q^m g)]
+ \\   &+ \lambda(\Q^n g).
\end{align}
To see this, notice first that  the statement is trivial for $n=1$. Assuming it
true for some $n \ge 1$, we prove the same thing with $n+1$ in place of $n$.
Apply \eqref{H7.1} with $\Lambda = \Gamma_0$ and $\Q^n g$ for $n$  odd in place of $g$:
\[
\E^{\Gamma_0}[\lambda(\Q^n g)] = \ent_{\E^{\Gamma_0}}(\Q^n g) + \lambda(\E^{\Gamma_0} \Q^n g),
\]
and, again from \eqref{H7.1} with $\Lambda = \Gamma_1$ and
$ \Q^n g$ for $n$  even in place of $g$,
\[
\E^{\Gamma_1}[\lambda(\Q^n g)] = \ent_{\E^{\Gamma_1}}(\Q^n g) + \lambda(\E^{\Gamma_1} \Q^n g).
\]
From the last two displays, for odd $n$ we get  
\begin{align*}
\E^{\Gamma_1}[\lambda(\Q^n g) ] = \E^{\Gamma_1}[\ent_{\E^{\Gamma_0}}(\Q^n g)]
+ \ent_{\E^{\Gamma_1}}(\E^{\Gamma_0} \Q^n g) + \lambda(\Q^{n+1} g)
\end{align*}
while for $n$ even  
\begin{align*}
\E^{\Gamma_0}[\lambda(\Q^n g)] = \E^{\Gamma_0}[\ent_{\E^{\Gamma_1}}(\Q^n g)]
+ \ent_{\E^{\Gamma_0}}(\E^{\Gamma_1} \Q^n g) + \lambda(\Q^{n+1} g).
\end{align*}
Using these, and  applying $\E^{\Gamma_0}$ or $\E^{\Gamma_1}$ to \eqref{H7.2} when $n$ is even or odd respectively, we readily obtain
\eqref{H7.2} with $n+1$ in place of $n$. This shows the
veracity of \eqref{H7.2}.
Using Lemma \ref{conv}, we have $\Q^n[\lambda(g)] \to \nu[\lambda(g)]$
and $\lambda^n(\Q^n g) \to \nu[g]$, $\nu$-a.e. From this
and Fatou's lemma, \eqref{H7.2} gives
\begin{align}
\ent_\nu(g) &\le \liminf_{n \to \infty}
\bigg\{
\nu\bigg[
\sum_{m=0,\ m\ \text{odd}}^{n-1} \Q^{n-m-1}  [\ent_{\E^{\Gamma_0}}(\Q^m g)]
+ \sum_{m=0,\ m\ \text{even}}^{n-1} \Q^{n-m-1} [\ent_{\E^{\Gamma_1}}(\E^{\Gamma_0}\Q^m g)]
\bigg]
\bigg\},        \nonumber
\\
&= \liminf_{n \to \infty}\bigg\{
\sum_{^{m=0,\ m\ \text{odd}} }^{n-1} \nu[\ent_{\E^{\Gamma_0}}(\Q^m g)]
+ \sum_{m=0,\ m\ \text{even}}^{n-1} \nu[\ent_{\E^{\Gamma_1}}(\E^{\Gamma_0}\Q^m g)]
\bigg\}\label{H7.3}
\end{align}
where we used the fact that $\nu$ is a Gibbs  measure to obtain
the last equality.
Let $g=f^2$ and   apply Proposition \ref{LSGamma0} to bound the entropy
 \begin{align*}
\nu[\ent_{\E^{\Gamma_0}}(\Q^{m} f^2)]
&\le  \tilde C\nu \|\nabla_{\Gamma_0} \sqrt{\Q^m f^2}\|^2
\le  \tilde C [C_1 C_2^{m-1}  \nu \|\nabla_{\Gamma_1} f\|^2 
+ \tilde CC_2^{m} \nu \|\nabla_{\Gamma_0} f\|^2 ]
\\
\nu[\ent_{\E^{\Gamma_1}}(\E^{\Gamma_0}\Q^{m} f^2)]
&\le  \tilde C \nu \|\nabla_{\Gamma_1} \sqrt{\E^{\Gamma_0}\Q^m f^2}\|^2
\le \tilde C [C_1 C_2^{m-1}  \nu \|\nabla_{\Gamma_0} f\|^2 
+ \tilde CC_2^{m} \nu \|\nabla_{\Gamma_1} f\|^2 ],
\end{align*}
for $m$ odd and even respectively,  where,  for the last inequalities we used Lemma \ref{sweepGam2} and induction.
Substituting in \eqref{H7.3}, we obtain (recall that $0< C_2 < 1$)
\[
\ent_\nu(f^2) \le 
\frac{ \tilde C (C_{1}C_{2}^{-1}+C_2)}{1-C_2} \nu \|\nabla_{\Gamma_1} f\|^2 
+ \frac{ \tilde C (C_{1}C_{2}^{-1}+C_2)}{1-C_{2}} \nu \|\nabla_{\Gamma_0} f\|^2 
\le \overline{C} \, \nu \|\nabla f\|^2,
\]
where $\overline C$ is the largest of the two coefficients.
This  is the log-Sobolev inequality for $\nu$. 
\end{proof}
 
\section{Example}\label{sectExample1}
We consider the Hamiltonian for a measure on the Heisenberg group defined as in (\ref{HGexample}). Theorem \ref{thmExample} follows from the main result presented in Theorem \ref{thmGENERAL}. Thus, we need to verify that the conditions of Theorem \ref{thmGENERAL} are satisfied for a local specification with a Hamiltonian as in (\ref{HGexample}).

At first, we need to verify that the main hypothesis, that  the  single site measure without interactions  (consisting only of the phase) 
$\mu(dx)=\frac{e^{-\phi(x)}dx}{\int e^{-\phi(x)} dx}$  satisfies
the log-Sobolev inequality.  In our example where $\phi(x)=\dd^{p}(x_i) +c\dd^{p-1}(x_i)cos(\dd(x_i)), p>2$,
as explained in the introduction in section \ref{introHeis}, this is true, since the family of measures (\ref{introUb2}) satisfies the log-Sobolev inequality, a result that has been  proven in \cite{H-Z}. Furthermore, hypothesis (\ref{H2.3}) and (\ref{H2.4}) about  the  Carnot-Carath\'eodory distance on the Heisenberg group $\He_1$ are true (see \cite{Mo} and \cite{H-Z}).    

At first one notices, that   for   convenience    the interaction potential can be written in the  following  equivalent form: 
\begin{equation}V(x,\omega)=  \delta \dd^{r }(x) +\sum _{k=1}^{r-1}a_k\dd^{r-k}(x) \dd^{k}(\omega) \label{eqform}\end{equation}where $a_k= \begin{pmatrix}r \\
k \\
\end{pmatrix}$ the  binomial    coefficients.

 For conditions (\ref{lowerH}) - (\ref{Hgeo}),  the first one easily follows from $\dd^{r}(xz)\leq 2^{r-1}\dd(x)+2^{r-1}\dd( z)$ for every $r\in\mathbb{N}$ and the specific form of $\phi$ and $V$. The second and third follows  from   $\dd (x^{-1})= \dd(x) $  and $\dd (\gamma (s))\leq \dd (z)$ for any geodesic from $0$ to $z$, both by the definition of the Carnot-Carath\'eodory  distance.

  Finally, conditions (\ref{H2.1})-(\ref{H2.2}) and (\ref{H2.5})-(\ref{H2.9}) can easily be verified for any $s=2p-2$ and $r\leq \frac{p+2}{2}$, if one writes     the interaction potential in the   form (\ref{eqform}).

\bibliography{mybibfile} 

\end{document}